\documentclass[12pt,reqno]{amsart}

\usepackage[margin=1.2in]{geometry}  
\usepackage{graphicx}              
\usepackage{amsmath}               
\usepackage{amsfonts}              
\usepackage{amsthm}                
\usepackage{amssymb}
\usepackage{appendix}
\usepackage{longtable}
\usepackage{todonotes}  
\usepackage{enumerate}
\usepackage{mathtools}
\usepackage{floatrow}

\newtheorem{thm}{Theorem}[section]

\newtheorem{lem}[thm]{Lemma}

\newtheorem{prop}[thm]{Proposition}

\newtheorem{que}[thm]{Question}
\numberwithin{equation}{section}
\theoremstyle{remark}
\newtheorem{remark}{Remark}

\newcommand{\bd}[1]{\mathbf{#1}}  
\newcommand{\RR}{\mathbb{R}}      
\newcommand{\ZZ}{\mathbb{Z}}      
\newcommand{\CC}{\mathbb{C}}

\newcommand{\mat}[1]{\left(\begin{matrix} #1 \end{matrix} \right)}  
\newcommand{\al}[1]{\begin{align}#1\end{align}}
\newcommand{\aln}[1]{\begin{align*}#1\end{align*}}

\newcommand{\HH}{\mathbb{H}}
\newcommand{\jj}{\bd{j}}
\newcommand{\bs}{\backslash}
\newcommand{\br}{\text{BR}}
\newcommand{\bms}{\text{BMS}}
\newcommand{\ps}{\text{PS}}
\begin{document}
\raggedbottom

\title{Statistical Regularity of Apollonian Gaskets}
\author[Xin Zhang]{Xin Zhang}
\address{%
Department of Mathematics, University of Illinois, 1409 W Green Street, Urbana, IL 61801
}
\email{xz87@illinois.edu}
\begin{abstract} 
Apollonian gaskets are formed by repeatedly filling the gaps between three mutually tangent circles with further tangent circles. In this paper we give explicit formulas for the the limiting pair correlation and the limiting nearest neighbor spacing of centers of circles from a fixed Apollonian gasket. These are corollaries of the convergence of moments that we prove.  The input from ergodic theory is an extension of Mohammadi-Oh's Theorem on the equidisribution of expanding horospheres in infinite volume hyperbolic spaces.

\end{abstract}

\maketitle
\section{Introduction}
\subsection{Introduction to the problem and statement of results}
Apollonian gaskets, named after the ancient Greek mathematician Apollonius of Perga (200 BC), are fractal sets formed by starting with three mutually tangent circles and iteratively inscribing new circles into the curvilinear triangular gaps (see Figure \ref{construction}). \par
\begin{figure}[h]
\begin{center}
\includegraphics[scale=0.3]{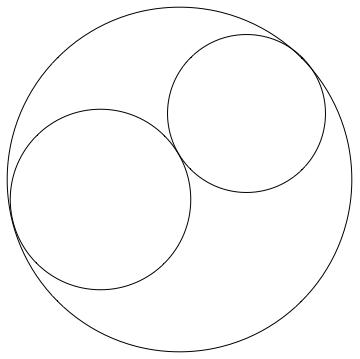}
\includegraphics[scale=0.3]{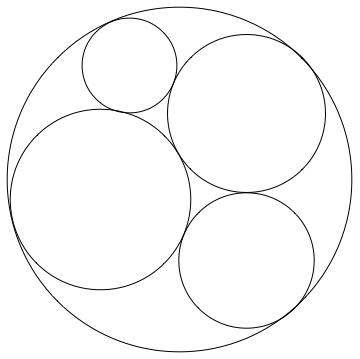}
\includegraphics[scale=0.3]{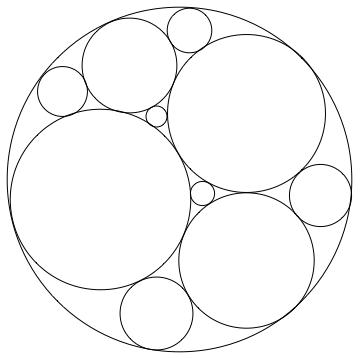}
\end{center}
\caption{Construction of an Apollonian gasket}
\label{construction}
\end{figure}
The last 15 years have overseen tremendous progress in understanding the structure of Apollonian gaskets from different viewpoints, such as number theory and geometry \cite{GLMWY03}, \cite{Fu10}, \cite{BF12}, \cite{BK14}, \cite{KO11}, \cite{OS12}.  In the geometric direction, generalizing a result of \cite{KO11}, Hee Oh and Nimish Shah proved the following remarkable theorem concerning the growth of circles. \par
Place an Apollonian gasket $\mathcal{P}$ in the complex plane $\mathbb C$.  Let $\mathcal P_t$ be the set of circles from $\mathcal P$ with radius greater than $e^{-t}$, and let $\mathcal C_t$ be the set of centers from $\mathcal P_t$.  Oh-Shah proved:
\begin{thm}[Oh-Shah, Theorem 1.6, \cite{OS12}]\label{0827}There exists a finite Borel measure $v$ supported on $\overline{\mathcal P}$, such that for any open set $E \subset \mathbb C$ with boundary $\partial E$ empty or piecewise smooth (see Figure \ref{aponr}), the cardinality $N(E,t)$ of the set $\mathcal C_t\cap E$, satisfies 
$$\lim_{t\rightarrow \infty}\frac{N(E,t)}{e^{\delta t}}=v(E),$$
where $\delta\approx 1.305688$ \cite{Mc98} is the Hausdorff dimension of any Apollonian gasket.
\end{thm}
\begin{figure}[!h]
   \includegraphics[scale = 0.4]{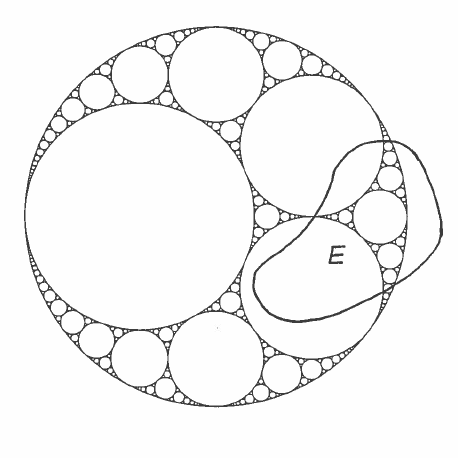}{\caption{A region $E$ with piecewise smooth boundary\label{aponr}}\label{center}}
\end{figure}
Theorem \ref{0827} gives a satisfactory explanation on how circles are distributed in an Apollonian gasket in large scale. In this paper we study some questions concerning the fine scale distribution of circles, for which Theorem \ref{0827} yields little information. For example, one such question is the following.
\begin{que}\label{que1}Fix $\xi>0$.  How many circles in $\mathcal P_t$ are within distance $\xi/e^t$ of a random circle in $\mathcal P_t$?
\end{que}
Here by distance of two circles we mean the Euclidean distance of their centers. Question \ref{que1} is closely related to the pair correlation of circles. In this article, we study the pair correlation and the nearest neighbor spacing of circles, which concern the fine structures of Apollonian gaskets. In particular, Theorem \ref{pairthm} gives an asymptotic formula for one half of the expected number of circles in Question \ref{que1},  as $t\rightarrow \infty$.  \par

Let $E\subset \CC$ be an open set with $\partial E$ empty or piecewise smooth as in Theorem \ref{0827}, and with $E\cap \mathcal P\neq\emptyset$ (or equivalently, $v(E)>0$).  This is our standard assumption for $E$ throughout this paper.  The pair correlation function $P_{E,t}$ on the growing set $\mathcal C_t$ is defined as
\al{\label{1910}
P_{E,t}(\xi):=\frac{1}{2\#\{\mathcal C_t\cap E\}}\sum_{\substack{p, q\in\mathcal C_t\cap E\\q\neq p}}\bd{1}\{e^{t}|p-q|<\xi\},
}
where $\xi\in(0,\infty)$ and $|p-q|$ is the Euclidean distance between $p$ and $q$ in $\mathbb C$. We have a factor $1/2$ in the definition \eqref{1910} so that each pair of points is counted only once. \par

For any $p\in\mathcal C_t$, let $d_t(p)=\min\{|q-p|:q\in\mathcal C_t,q\neq p\}$.
The nearest neighbor spacing function $Q_{E,t}$ is defined as
\al{\label{1911}
Q_{E,t}(\xi):=\frac{1}{\#\{\mathcal C_t\cap E\}}\sum_{\substack{p\in\mathcal C_t\cap E}}\bd{1}\{e^td_t(p)<\xi\}.
}

For simplicity we abbreviate $P_{E,t},Q_{E,t}$ as $P_t,Q_t$ if $E=\mathbb C$.  It is noteworthy that in both definitions \eqref{1910} and \eqref{1911}, we normalize distance by multiplying by $e^t$.  The reason can be seen in two ways.  First, Theorem \ref{0827} implies that a random circle in $\mathcal C_t$ has radius $\asymp e^{-t}$, so a random pair of nearby points (say, the centers of two tangent circles) from $\mathcal C_t$ has distance $\asymp e^{-t}$, thus $e^{-t}$ is the right scale to measure the distance of two nearby points in $\mathcal C_t$.  The second explanation is more informal:  if $N$ points are randomly distributed in the unit interval $[0,1]$, then a random gap is of the scale $N^{-1}$; more generally, if $N$ points are randomly distributed in a compact $n$-fold, the distance between a random pair of nearby points should be of the scale $N^{-1/n}$.  In our situation, as $t\rightarrow\infty$, the set $\mathcal{C}_t$ converges to $\overline{\mathcal P}$, where $\overline{\mathcal{P}}$ has Hausdorff dimension $\delta\approx 1.305688$.  From Theorem \ref{0827}, we know that $\#\mathcal C_t\asymp e^{\delta t}$, so our scaling $e^{-t}$ agrees with the heuristics that the distance between two random nearby points in $\mathcal C_t$ should be $(e^{\delta t})^{-\frac{1}{\delta}}=e^{-t}$.  

Before stating our main results, we introduce terminology. It is convenient for us to work with the upper half-space model of the hyperbolic 3-space $\mathbb H^3$:
$$\mathbb H^3=\{z+r\bd{j}: z=x+y\bd{i}\in \mathbb C, r\in\mathbb R\}.$$  
We identify the boundary $\partial \mathbb H^3$ of $\mathbb H^3$ with $\mathbb C\cup\{\infty\}$. For $q=x+y\bd{i}+r\jj\in\HH^3$, we define $\Re(q)=x+y\bd{i}$ and $\Im(q)=r$.\par
Let $G=PSL(2,\mathbb C)$ be the group of orientation-preserving isometries of $\mathbb H^3$.  We choose a discrete subgroup $\Gamma< PSL(2,\mathbb C)$ whose limit set $\Lambda(\Gamma)=\overline{\mathcal P}$ such that $\Gamma$ acts transitively on circles from $\mathcal{P}$.  It follows from Corollary 1.3, \cite{BJ97} that $\Gamma$ is geometrically finite.\par

Without loss of generality, we can assume that the bounding circle of $\mathcal{P}$ is $C(0,1)$, where $C(z,R)\subset \mathbb C$ is the circle centered at $z$ with radius $R$.
Let $S\subset \mathbb H^3$ be the hyperbolic geodesic plane with $\partial S=C(0,1)$, and $H< PSL(2,\mathbb C)$ be the stabilizer of $S$. \par
As an isometry on $\mathbb H^3$, each $g\in PSL(2,\mathbb C)$ sends $S$ to a geodesic plane, which is either a vertical plane or a hemisphere in the upper half-space model of $\mathbb H^3$.  We define continuous maps $\bd{q}: G\rightarrow\overline{\HH^3}$, $\bd{q}_\Re: G \rightarrow\widehat{\CC}$ as follows:  
\al{\label{1244}&\bd{q}(g):=\begin{cases}\text{the apex of }g(S),&\text{if }\infty \not\in g(\partial S),\\
\infty, & \text{if } \infty\in g(\partial S),
\end{cases} \\
\label{1245}&\bd{q}_{\Re}(g):=\begin{cases}\Re(\bd{q}(g)),&\text{if }\infty \not\in g(\partial S),\\
\infty, & \text{if } \infty\in g(\partial S).
\end{cases}}
 
We further define a few subsets of $\mathbb H^3$.  For $\xi>0$, let $B_{\xi}:=\{z\in\mathbb C: |z|<\xi\}$ and let $B_{\xi}^{*}\subset \mathbb{H}^3$ be the ``infinite chimney'' with base $B_{\xi}$, where for any $\Omega\subset \mathbb C$, 
\al{\Omega^*:=\{z+r\bd{j}:z\in\Omega, r\in (1,\infty) \}.\label{0530}}
Let $\mathfrak C_\xi$ be the cone in $\HH^3$:
\al{\label{1010}\mathfrak{C}_{\xi}:=\left\{z+r\bd{j}\in \HH^3: \frac{r}{|z|}>\frac{1}{\xi},0<r\leq 1\right\}.}
Now we can state our main theorems.
\begin{thm}[limiting pair correlation]\label{pairthm} For any open set $E\subset \mathbb C$ with $E\cap \mathcal P\neq\emptyset$ and $\partial E$ empty or piecewise smooth, there exists a continuously differentiable function $P$ independent of $E$, supported on $[c,\infty)$ for some $c>0$, such that 
$$\lim_{t\rightarrow\infty} P_{E,t}(\xi)=P(\xi).$$
The derivative $P'$ of $P$ is explicitly given by 
$$P'(\xi)=\frac{\delta}{2\mu_H^{\normalfont\text{PS}}(\Gamma_H\bs H)}\int_{h\in\Gamma_H\bs H}\sum_{\substack{\gamma\in\gamma_H\bs(\Gamma-\Gamma_H)\\\bd{q}(h^{-1}\gamma^{-1})\in B_{\xi}^*\cup \mathfrak C_\xi}}\frac{\vert\bd{q}_{\Re}(h^{-1}\gamma^{-1})\vert^{\delta}}{\xi^{\delta+1}}d\mu_{H}^{\normalfont\ps}(h).$$
\end{thm}

Here $\Gamma_H:=\Gamma\cap H$, and $\mu_H^{\ps}$ is a Patterson-Sullivan type measure on $H$.  Besides $\mu_H^{\ps}$, we will also encounter other conformal measures $\mu_{N}^{\ps}, w, m^{\br},m^{\bms}$, which are built on the Patterson-Sullivan densities. The measure $\mu_{N}^{\ps}$ is a Patterson-Sullivan type measure on the horospherical group $N:=\left\{n_z=\mat{1&z\\0&1}:z\in\mathbb C\right\}$, $w$ is the pullback measure of $\mu_N^{\ps}$ on $\CC$ under the identification $z\rightarrow n_z$, and $m^{\br},m^{\bms}$ are the Burger-Roblin, Bowen-Margulis-Sullivan measures.  We will have a detailed discussion of these measures in Section \ref{0129}. \par

See Figure \ref{pair} and Figure \ref{deri} for some numerical evidence for Theorem \ref{pairthm}.  Let $\mathcal P(\theta_1,\theta_2)$ be the unique Apollonian gasket determined by the four mutually tangent circles $C_0,C_1,C_2,C_3$, where $C_0=C(0,1)$ is the bounding circle, and $C_1,C_2,C_3$ are tangent to $C_0$ at $1,e^{\theta_1\bd{i}}$, $e^{\theta_2\bd{i}}$.  Figure \ref{pair}, Figure \ref{wholehalf} and Figure \ref{deri} are based on the gasket $\mathcal P(\frac{1.8\pi}{3},\frac{3.7\pi}{3})$.  Figure \ref{figure_1} suggests that the limiting pair correlations for different Apollonian gaskets are the same.  The reason is twofold. First, for a fixed gasket, the limiting pair correlation locally looks the same everywhere. Second, one can take any Apollonian gasket to any other one by a M\"obius transformation, which locally looks like a dilation combined with a rotation, and it is an elementary exercise to check that the limiting pair correlation is invariant under these motions.  
\par

\begin{figure}[H]
\includegraphics[scale = 0.5]{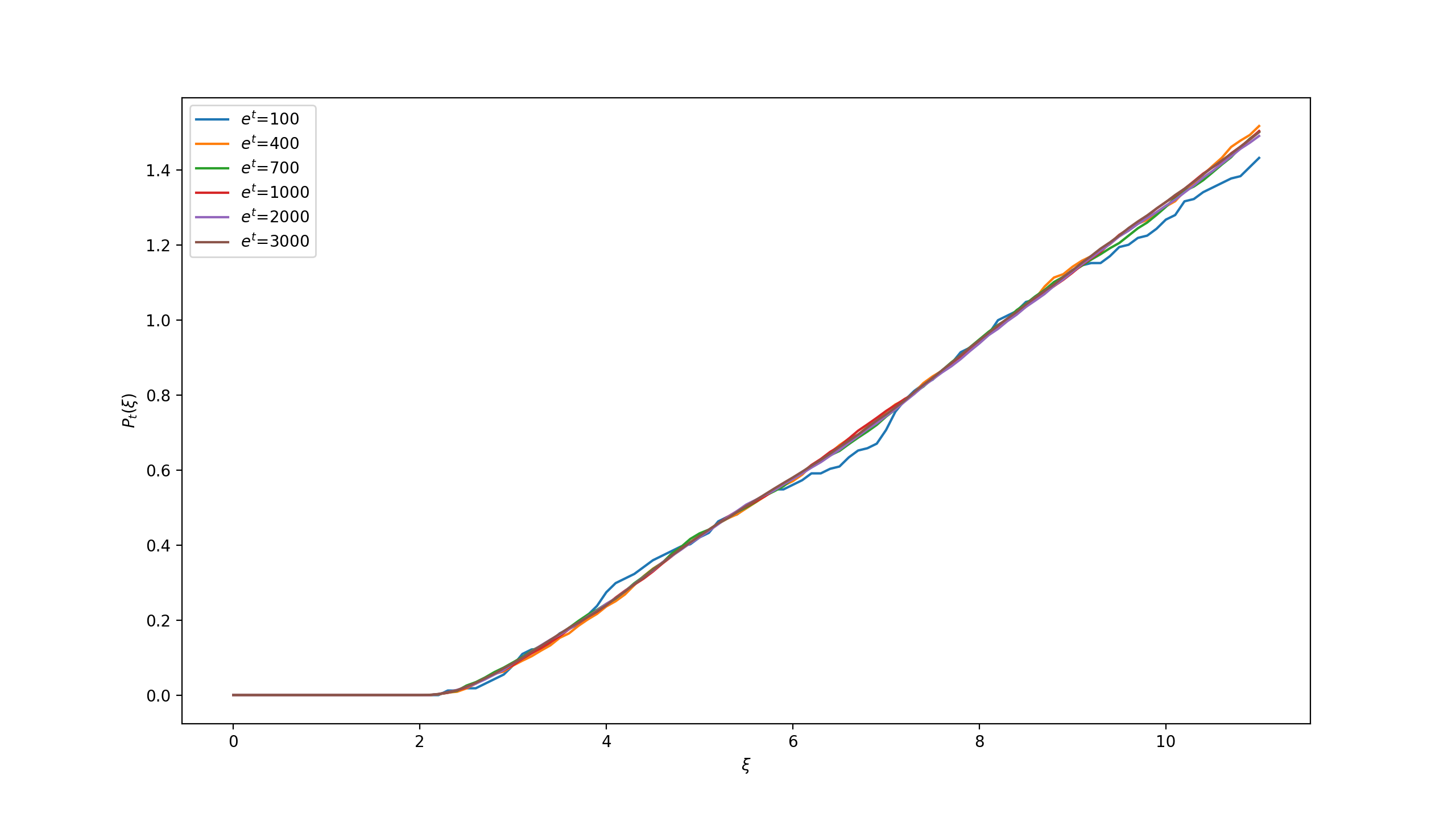}}{\caption{The plot for $P_t$ with various $t$'s}\label{pair}
\end{figure}

\begin{figure}[H]
\includegraphics[scale = 0.5]{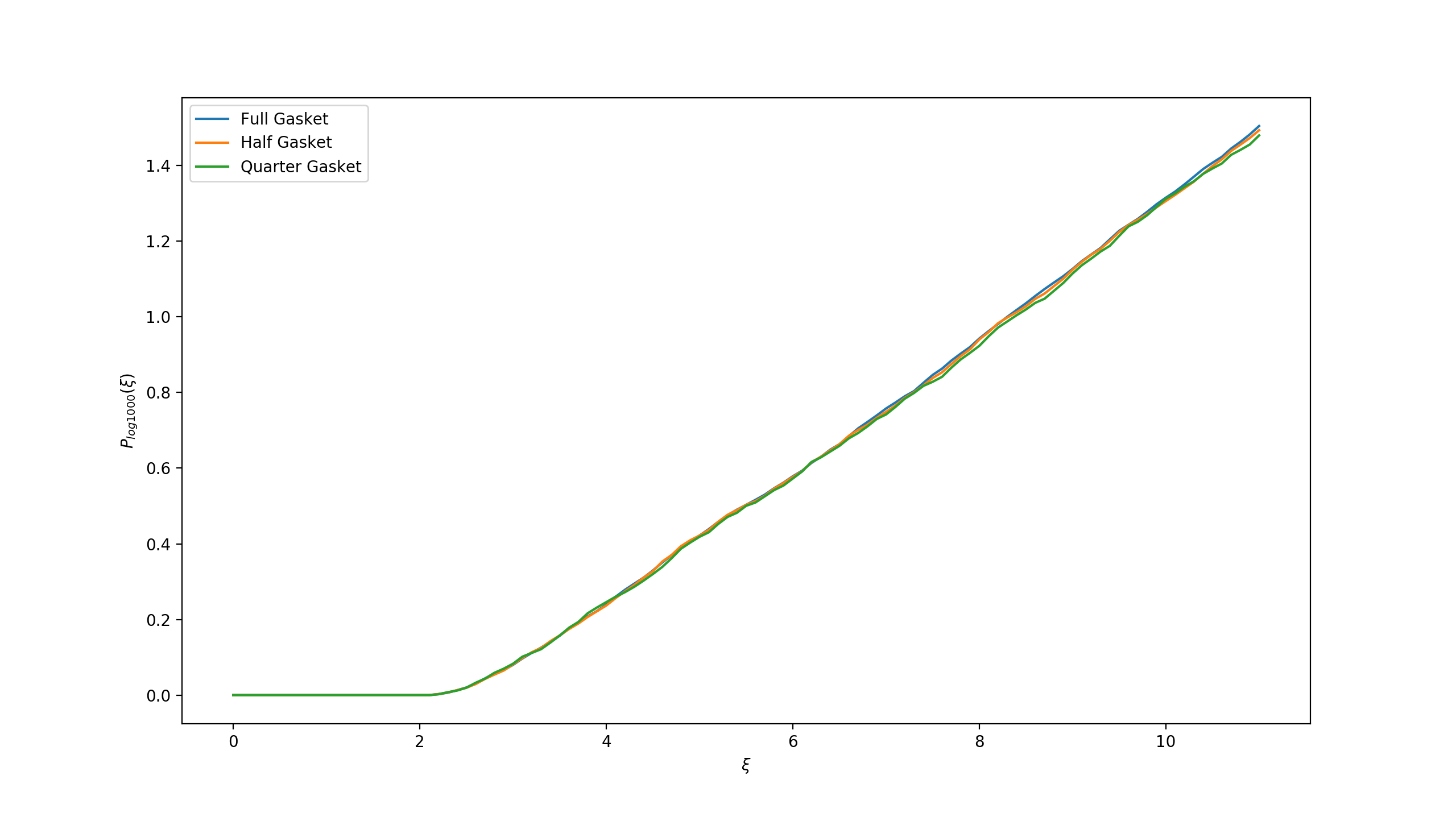}}{\caption{Pair correlations for the whole plane, half plane and the first quadrant}\label{wholehalf}
\end{figure}

\begin{figure}[H]
\includegraphics[scale = 0.5]{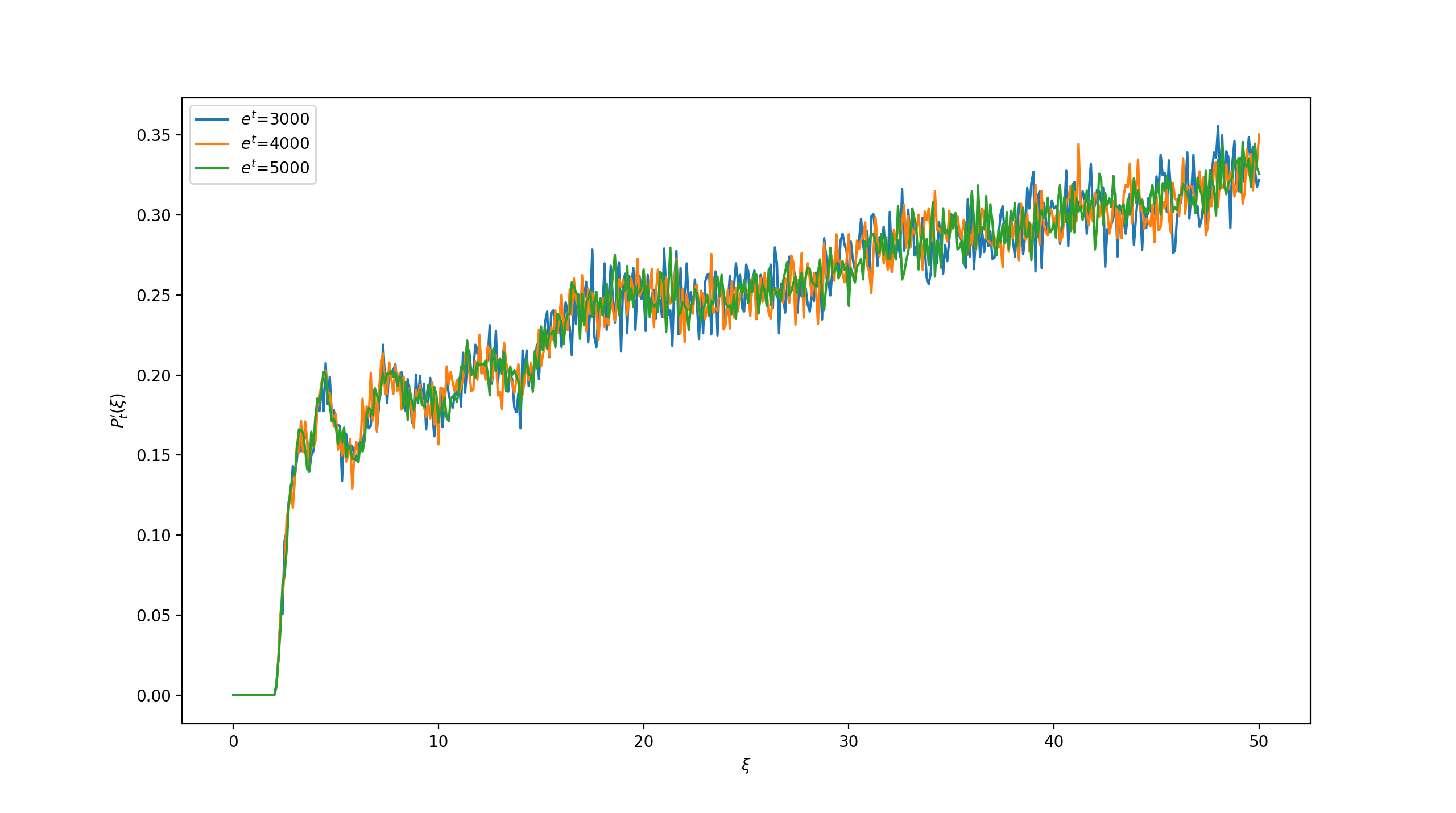}}{\caption{The empirical derivative $P_t'(\xi)$ for different $t$, with step=0.1}\label{deri}
\end{figure}

\begin{figure}[H]
\includegraphics[scale = 0.5]{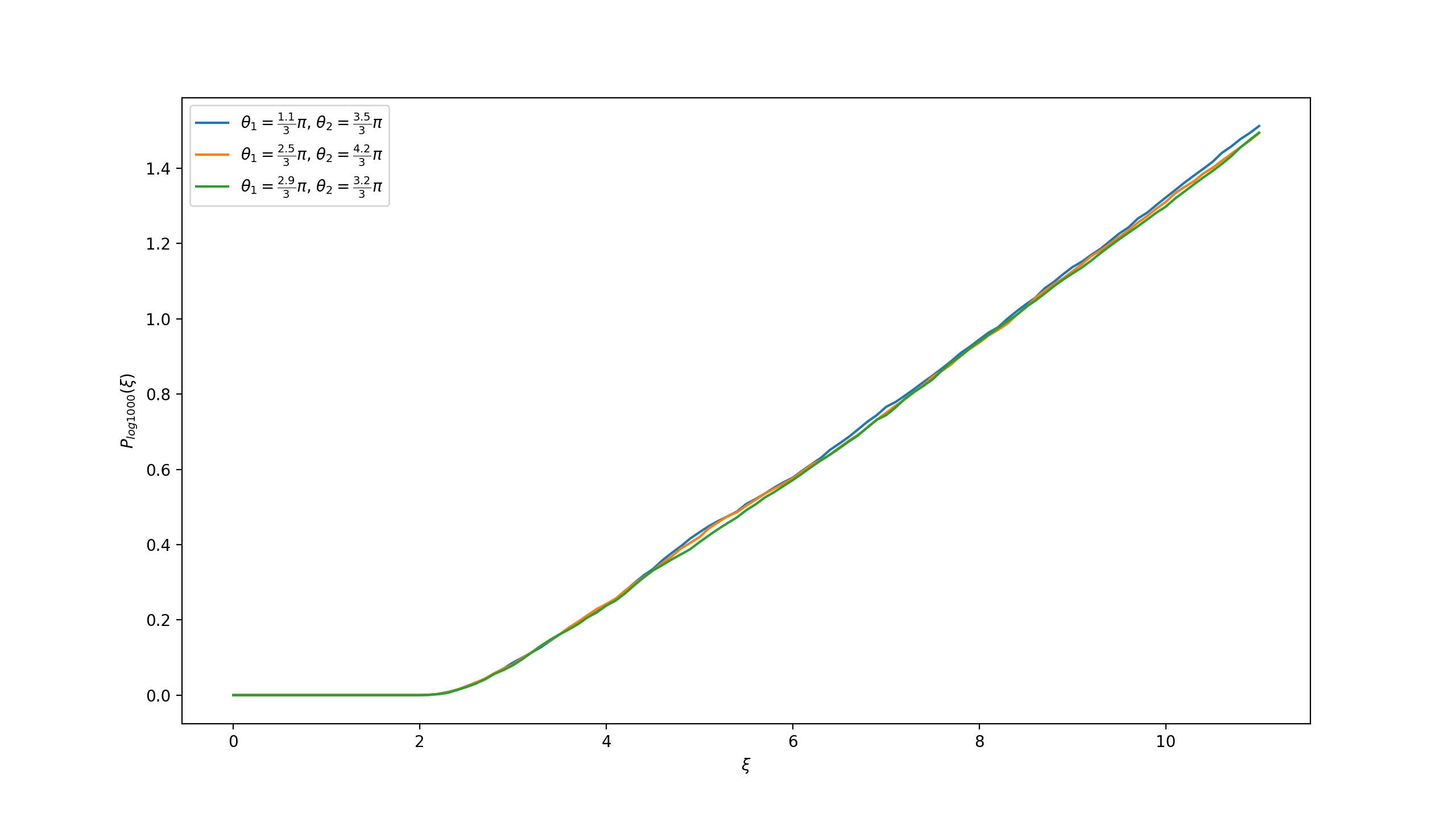}}{\caption{Pair correlation functions for different Apollonian gaskets}\label{figure_1}
\end{figure}

\begin{thm}[limiting nearest neighbor spacing]\label{nearthm}There exists a continuous function $Q$ independent of $E$, supported on $[c,\infty)$ for some $c>0$, such that 
\al{
\lim_{t\rightarrow\infty}Q_{E,t}(\xi)=Q(\xi).
} 

The formula for $Q$ is explicitly given by
\al{
Q(\xi)=1-\frac{\delta}{\mu_H^{\ps}(\Gamma_H\bs H)}\int_{\Gamma_H\bs H}\int_0^\infty e^{-\delta t}\bd{1}\{\#\bd{q}(a_{-t}h^{-1}(\Gamma-\Gamma_H)) \cap B_\xi^*=0\}dtd\mu_{H}^{\ps}(h).
}  
\end{thm} 
Here $a_{-t}$ is the diagonal matrix $\mat{e^{\frac{t}{2}}&0\\0&e^{-\frac{t}{2}}}$, and see Figure \ref{near} for numerical evidence.

\begin{figure}[H]
\includegraphics[scale=0.5]{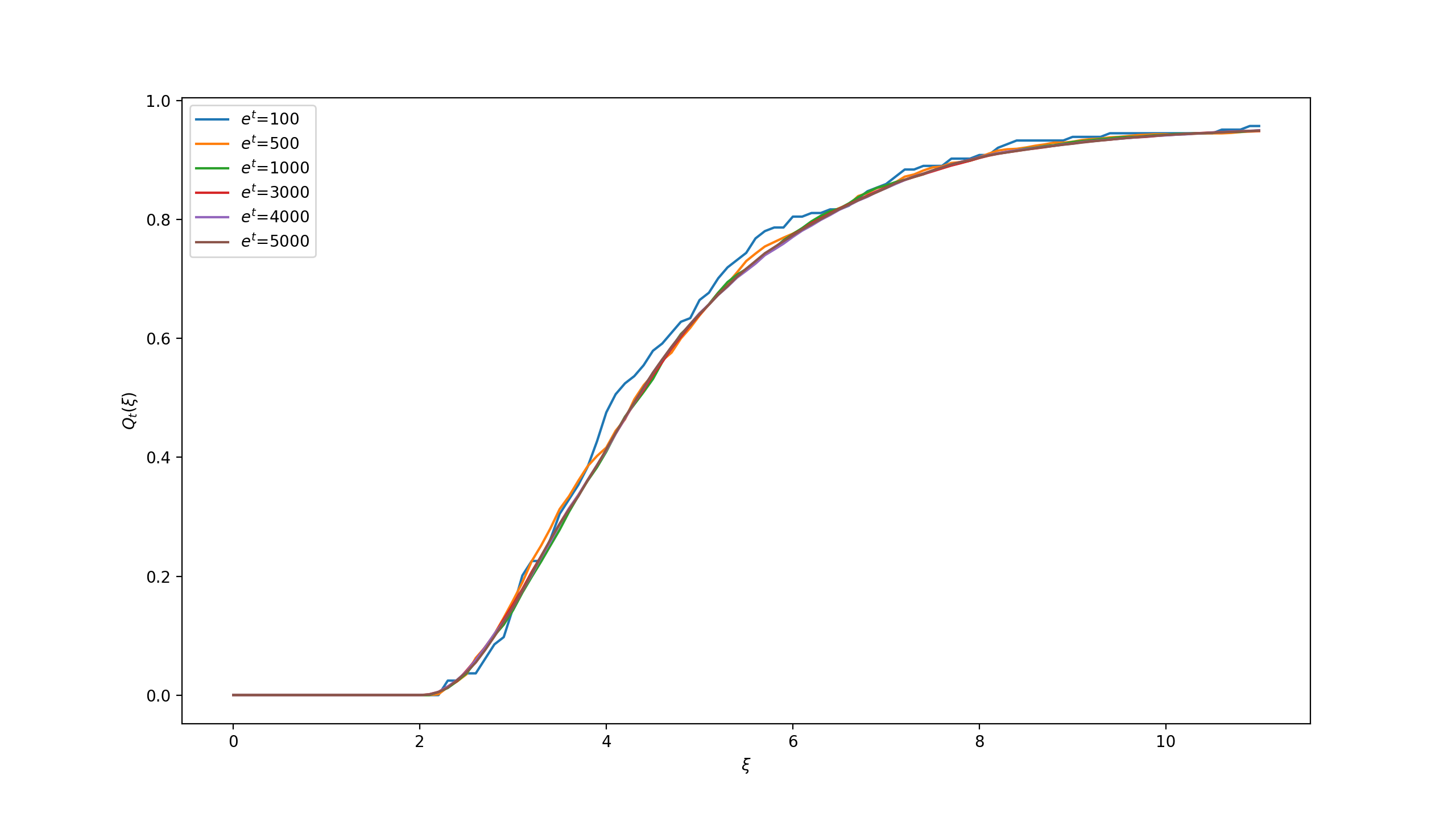}{\caption{The nearest neighbor spacing function $Q_t(\xi)$ for various $t$'s}\label{near}}
\end{figure}
\begin{remark}
Figure \ref{near} suggests that $Q$ should be differentiable.  Unlike the limiting pair correlation, we have not been able to prove the differentiability of $Q$ based on our formula for $Q$.
\end{remark}

Both Theorem \ref{pairthm} and Theorem \ref{nearthm} follow from the convergence of moments (Theorem \ref{07142}), which we explain now.\par
Let $\boldsymbol{\Omega}=\prod_{1\leq i\leq k}\Omega_i\subset \mathbb C^k $, where $\Omega_i,1\leq i\leq k$ are bounded open subsets of $\mathcal{C}$ with piecewise smooth boundaries. \par

For $z\in\mathbb C$, let $$\mathcal B_t(\Omega_i,z):=(e^{-t}\Omega_i+z)\cap\mathcal C_t,$$ and $$\mathcal N_t(\Omega_i,z):=\#\mathcal B_t(\Omega_i,z).$$ 

Let $\bd{r}=\langle r_1,\dots, r_k\rangle, \boldsymbol{\beta}=\langle \beta_1,\dots,\beta_k\rangle$ be multi-indices, where $r_i\in\ZZ_{\geq 0},\beta_i\in \RR_{\geq 0},1\leq i\leq k$, and at least one component of $\bd{r},\boldsymbol{\beta}$ is nonzero.
We want to understand the behaviors of the following two integrals
\al{\label{0819}\int_{\CC}\prod_{1\leq i\leq k}\bd{1}\{\mathcal N_t(\Omega_i,z)=r_i\}\chi_E(z)dz}
and 
\al{\label{0820}\int_{\CC}\prod_{1\leq i\leq k}\mathcal N_t(\Omega_i,z)^{\beta_i}\chi_E(z)dz,}
as $t\rightarrow \infty$, where $\chi_E$ is the characteristic function for an open set $E\subset \mathbb C$ with no boundary or piecewise smooth boundary.  Both \eqref{0819} and \eqref{0820} capture information about the correlation of centers. \par

Define functions $F_{\boldsymbol{\Omega},{\bd{r}}}, F_{\boldsymbol{\Omega}}^{\boldsymbol{\beta}}$ on $G$ by
\al{\label{03581}
&F_{\boldsymbol{\Omega},{\bd{r}}}(g):=\prod_{1\leq i\leq k }\bd{1}\left\{\#( \bd{q}(g^{-1}\Gamma/\Gamma_H)\cap \Omega_i^*)=r_i  \right\},\\
\label{03591}
&F_{\boldsymbol{\Omega}}^{\boldsymbol{\beta}}(g):=\prod_{1\leq i\leq k }\#( \bd{q}(g^{-1}\Gamma/\Gamma_H)\cap \Omega_i^*)^{\beta_i}.
}
We put inverse signs over $g$ in the definitions \eqref{03581} and \eqref{03591} so that both $F_{\boldsymbol{\Omega},\bd{r}}$ and $F_{\boldsymbol{\Omega}}^{\boldsymbol{\beta}}$ are left $\Gamma$-invariant functions and can be thought of as functions on $\Gamma\backslash G$.  \par

The following theorem holds:

\begin{thm}[convergence of moments]\label{07142} With notation as above, we have
\aln{\lim_{t\rightarrow\infty}e^{(2-\delta)t}\int_{\CC}\prod_{1\leq i\leq k}\bd{1}\{\mathcal N_t(\Omega_i,z)=r_i\}\chi_E(z)dz=\frac{m^{\normalfont\br}(F_{\boldsymbol{\Omega},\boldsymbol{r}})w(E)}{m^{\normalfont\bms}(\Gamma\bs G)},
}
and
\aln{\lim_{t\rightarrow\infty}e^{(2-\delta)t}\int_{\CC}\prod_{1\leq i\leq k}\mathcal N_t(\Omega_i,z)^{\beta_i}\chi_E(z)dz=\frac{m^{\normalfont\br}(F_{\boldsymbol{\Omega}}^{\boldsymbol{\beta}})w(E)}{m^{\normalfont\bms}(\Gamma\bs G)}.
}
\end{thm}

\subsection{An overview of the method}
To prove Theorem \ref{07142}. we first turn the integrals \eqref{0819} and \eqref{0820} into forms that fit into Mohammadi-Oh's theorem on the equidistribution of expanding horospheres (Theorem \ref{03211}). Here in particular, for our convenience we use the $HAN$ and $NAH$ decompositions for $G$.  Here $H, A, N$ are certain subgroups of $G$ (see Section \ref{hyperbolic} for the definitions of $H,A$ and $N$). These decompositions seem new to us and we name them \emph{the generalized Iwasawa decompositions}.  
\begin{thm}[Mohammadi-Oh, Theorem 1.7, \cite{MO15}]\label{03211} Suppose $\Gamma<G$ is geometrically finite.  Suppose $\Gamma\backslash \Gamma N$ is closed in $\Gamma\backslash G$ and $|\mu_{N}^{\text{PS}}|<\infty$.  For any $\Psi\in C_c^{\infty}(\Gamma\backslash G)$ and any $f\in C^{\infty}(\Gamma\backslash \Gamma{N})$, we have
\al{\label{1258}\lim_{t\rightarrow\infty}e^{(2-\delta)t}\int_{\Gamma\backslash\Gamma N}\Psi(na_t)f(n)d\mu_{N}^{\normalfont \text{Leb}}(n)=\frac{m^{\rm{BR}}(\Psi)\mu_{N}^{\rm{PS}}(f)}{m^{\rm{BMS}}(\Gamma\backslash G) }.}
\end{thm}

However, Theorem \ref{03211} can not be directly applied, because in the statement of Theorem \ref{03211},  the test function $\Psi$ is assumed to be compactly supported and smooth, while in our situation, $\Psi$ is $F_{\boldsymbol{\Omega},\boldsymbol{r}}$ or $F_{\boldsymbol{\Omega}}^{\boldsymbol{\beta}}$, which are neither continuous nor compactly supported.  The smoothness condition for $f$ and $\Psi$ is for the purpose of obtaining a version of equidistribution with exponential convergence rate.  This is not needed for our purpose, as we only pursue asymptotics.  By the same method from \cite{OS13}, the restriction for $f$ can be relaxed to be in $L^1(\Gamma\bs \Gamma N)$ together with some mild regularity assumption, and $\Psi$ can be relaxed to be continuous and compactly supported; but this is still not enough for our purpose.  We circumvent this technical difficulty by proving Proposition \ref{1426}, illustrating some hierarchy structure in the space $\mathcal W$ of pairs of test functions $(f,\Psi)$ where the conclusion of Theorem \ref{03211} holds.

Theorem \ref{0827} implies that certain pairs $(f_0,\Psi_0)$ related to counting circles are in the space $\mathcal W$.  An elementary geometric argument shows that $F_{\boldsymbol{\Omega}, \boldsymbol{r}},F_{\boldsymbol{\Omega}}^{\boldsymbol{\beta}}$ are dominated by $\Psi_0$.  This together with Proposition \ref{1426} give us the desired Theorem \ref{07141}, which is an extension of Theorem \ref{03211}. \par

\begin{thm}\label{07141} Let $\Gamma< PSL(2,\CC)$ be a discrete group with the limit set $\Lambda(\Gamma)=\overline{\mathcal{P}}$ and acting transitively on the circles from $\mathcal{P}$.  Let $\Psi=F_{\bd{\Omega},\bd{r}}$ or $F_{\boldsymbol{\Omega}}^{\boldsymbol{\beta}}$, where $F_{\bd{\Omega},\bd{r}}$ and $F_{\boldsymbol{\Omega}}^{\boldsymbol{\beta}}$ are defined by \eqref{03581} and \eqref{03591}. Then $m^{\normalfont\br}(\Psi)<\infty$, and
\al{\lim_{t\rightarrow\infty}e^{(2-\delta)t}\int_{\CC}\chi_E(z)\Psi(n_za_t)dz=\frac{m^{\normalfont\br}(\Psi)w(E)}{m^{\normalfont\bms}(\Gamma\bs G)}
.}
\end{thm}

Theorem \ref{07142} then follows from Theorem \ref{07141}.
\par

\begin{remark}\raggedbottom
It is desirable to prove a version of Theorem \ref{03211} only assuming the integrality of $\Psi$ over the Burger-Roblin measure  plus some mild restriction.  While it is an exercise to relax the compactly-supported assumption to being in $L^1$ when the hyperbolic space has finite volume, such an extension seems much less obvious (at least to the author) if the space has infinite volume.  We have made partial progress (say, $\Psi$ can be in the Schwartz space) but haven't been able to achieve sufficient generality to encompass Theorem \ref{07141}.    
\end{remark}

\subsection{A historical note}
Pair correlation as well as other spatial statistics have been widely used in various disciplines such as physics and biology. For instance, in microscopic physics, the Kirkwood-Buff Solution Theory \cite{BK51} links the pair correlation function of gas molecules, which encodes the microscopic details of the distribution of these molecules to some macroscopic thermodynamical properties of the gas such as pressure and potential energy. In macroscopic physics, cosmologists use pair correlations to study the distribution of stars and galaxies.  \par
Within mathematics, there is also a rich literature on the spatial statistics of point processes arising from various settings, such as Riemann zeta zeros \cite{Mo73}, fractional parts of $\{\sqrt{n},n\in\ZZ^+\}$ \cite{EM04}, directions of lattice points \cite{BZ06}, \cite{BPZ14}, \cite{KK15}, \cite{RS14}, \cite{MV14}, \cite{BMV15}, Farey sequences and their generalizations \cite{Ha70}, \cite{BCZ01}, \cite{ABC01}, \cite{RZ15}, \cite{ACZ15}, \cite{ACMZ15}, and translation surfaces \cite{AC10}, \cite{ACL15}, \cite{UW16}.  Our list of interesting works here is far from inclusive.  These statistics can contain rich information and yield surprising discoveries.  For instance, Montgomery and Dyson's famous discovery that the pair correlation of Riemann zeta zeros agrees with that of the eigenvalues of random Hermitian matrices, bridges analytic number theory and high energy physics. \par
There is a major difference between all works mentioned above and our investigation of circles here. In the above works, the underlying point sequences are uniformly distributed in their ``ambient'' spaces.  In our case, the set of centers is fractal in nature: it is not dense in any reasonable ambient space such as $B_1$, the disk centered at 0 and of radius 1.  Consequently, we need different normalizations of parameters. \par
In some of the works above, the problems were eventually reduced to the equidistribution of expanding horospheres in finite volume hyperbolic spaces. In our case, we need an infinite volume version of this dynamical fact, which is Theorem \ref{03211}, as well as to take care of certain emerging issues in the infinite volume situation.  The main contribution of this paper, in the eyes of the author, is to introduce the recently rapidly developed theory of thin groups to study the fine scale structures of fractals, by displaying a thorough investigation of the well known Apollonian gaskets.  \par 

\subsection{The structure of the paper} Section \ref{hyperbolic} gives some basic background in hyperbolic geometry.  In Section \ref{setup} we set up the problem and reduce proving Theorem \ref{07142} to proving Theorem \ref{07141}.  In Section \ref{0129} we give a detailed discussion of some emerging conformal measures built up from the Patterson-Sullivan densities.  We finish the proof of Theorem \ref{07141} in Section \ref{equidstribution}.  Finally in Section \ref{pairnear} we explain how to deduce Theorem \ref{pairthm} and Theorem \ref{nearthm} from Theorem \ref{07141}.  We give complete detail for the limiting pair correlation; the limiting nearest neighbor spacing can be deduced in an analogous way and we sketch the proof.  

\subsection{Notation} We use the following standard notation.  The expressions $f\ll g$ and $f=O(g)$ are synonymous, and $f\asymp g$ means $f\ll g$ and $g\ll f$.  Unless otherwise specified, all the implied constants depend at most on the symmetry group $\Gamma$. The symbol $\bd{1}\{\cdot\}$ is the indicator function of the event $\{\cdot\}$. For a finite set $\mathcal S$, we denote the cardinality of $\mathcal S$ by $\#\mathcal S$.

\subsection{Ackowledgement}
Figures 3-7 were produced in a research project of Illinois Geometry Lab (IGL) \cite{CJKMZ}, where Weiru Chen, Calvin Kessler and Mo Jiao were the undergraduate investigators, Amita Malik was the graduate mentor, and the author of this paper was the faculty mentor.  Although we didn't use the results from \cite{MV14} directly, that paper together with the data produced from the IGL project gave us the main inspiration of this paper.  The technique employed in this paper is mainly from \cite{Wi15}, \cite{OS13}, \cite{MO15}. Thanks are also due to Prof. Curt McMullen for his enlightening comments and corrections.  
  
\section{Hyperbolic 3-space and groups of isometries}\label{hyperbolic}
We use the upper half-space model for the hyperbolic 3-space $\mathbb H^3$:
$$\mathbb H^3=\{x+y\bd{i}+r\bd{j}: x+y\bd{i}\in \mathbb C, r\in\mathbb R\}.$$  
The boundary $\partial \mathbb H^3$ of $\HH^3$ is identified with $\mathbb C\cup \{\infty\}$. \par

The hyperbolic metric and the volume form on $\mathbb{H}^3$ are given by
\aln{ds^2&=\frac{dx^2+dy^2+dr^2}{r^2},\\
dV&=\frac{dxdydr}{r^3}.}

Let $G=PSL(2,\CC)$ be the group of orientation-preserving isometries of $\HH^3$, and let $e$ be the identity element of $G$.  The action of $G$ on $\mathbb H^3$ is given explicitly as the following:

$$\mat{a&b\\c&d}(z+r\bd{j})= \frac{a\bar{c}|z|^2+a\bar{d}z+b\bar{c}\bar{z}+b\bar{d}+r^2a\bar{c}}{|cz+d|^2+r^2|c|^2}+\frac{r}{|cz+d|^2+r^2|c|^2}\bd{j}.$$
For any two points $q_1,q_2\in\HH^3$, the formula for their hyperbolic distance $d(q_1,q_2)$ is 
\al{\label{0518} d(q_1,q_2)=\text{Arccosh}\left( 1+ \frac{|q_1-q_2|^2}{2\Im(q_1)\Im(q_2)}\right),}
where $|q_1-q_2|$ is the Euclidean distance between $q_1$ and $q_2$.\par

Let $\pi_1,\pi_2$ be the maps from $G$ to $\text{T}^1(\HH^3),\HH^3$ defined by 
\aln{
&\pi_1(g):=g(X_1),\\
&\pi_2(g):=g(\jj).
}

The following subgroups of $G$ will appear in our analysis:

\begin{enumerate}[(i)]
\item $A=:\left\{a_t=\mat{e^{-\frac{t}{2}}&0\\0&e^{\frac{t}{2}}}:t\in\RR \right\}$.   
\item $K=:PSU(2)=\left\{\mat{a&b\\\bar{b}&\bar{a}}|a|^2+|b|^2 =1\right\}$.
\item $M=:\left\{ m_\theta=\mat{e^{\frac{\theta}{2}\bd{i}}&0\\0&e^{-\frac{\theta}{2}\bd{i}}}: \theta\in[0,2\pi) \right\}$.
\item $N=:\left\{ {n_z}=\mat{1&z\\0&1}: z\in\CC \right\}$.
\item $H:=SU(1,1)\cup SU(1,1)\mat{0&-1\\1&0}$, where \newline $SU(1,1)=\left\{\mat{\xi&\eta\\\bar{\eta}&\bar{\xi}}:\xi,\eta\in\CC, |\xi|^2-|\eta|^2=1 \right\}$.
\item $H_0:= SU(1,1)$, the identity component of $H$.
\item $\widetilde{A}=:\left\{\widetilde{a}_t=\mat{\cosh\frac{t}{2}&\sinh\frac{t}{2}\\\sinh\frac{t}{2}&\cosh \frac{t}{2}}:t\in\RR \right\}$.
\end{enumerate}\par
 
We now explain the geometric meaning of the above groups.  Let $\{X_1,X_2,X_3\}$ be an \emph{orthonomal frame} based at $\jj$,  where $X_1,X_2,X_3$ are unit vectors based at $\bd{j}$ pointing to the negative $r$ direction, positive $y$ direction, and the positive $x$ direction, respectively.  Let $S\subset \HH^3$ be the hyperbolic geodesic plane with boundary $\partial S=C(0,1)$, where $C(z,R)\subset \CC$ is the circle centered at $z$ with radius $R$.  The group $G$ can also be identified with the orthonormal frame bundle on $\HH^3$.  The flows $\{a_t(X_1):t\in\RR\},\{\widetilde{a}_t(X_3):t\in\RR\}$ are the geodesic flows containing $X_1,X_3$, respectively.  The group $H$ is the stabilizer of the geodesic plane $S$, $K$ is the stabilizer of $\jj$, and $M$ is the stabilizer of $X_1$.  The orbit $N(X_1)$ is the expanding horosphere containing $X_1$.  \par

In our analysis we adopt the following decomposition for $G$ which are particularly convenient for us:
\aln{
G=NAH; G=HAN.
}
We call these decompositions \emph{the generalized Iwasawa decompositions}.  \par
We further decompose the group $H$ via the Cartan decomposition: 
\al{\label{0532}
H=M\left(\widetilde{A}^{+}\cup \widetilde{A}^{+}\mat{0&-1\\1&0}\right)M,
}
where $$\widetilde{A}^{+}=:\left\{\widetilde{a}_t=\mat{\cosh\frac{t}{2}&\sinh\frac{t}{2}\\\sinh\frac{t}{2}&\cosh \frac{t}{2}}:t\in(0,\infty) \right\}.$$  
                                                                                                                                                                                                                                                  
For every $h\in H-M\cup \mat{0&-1\\1&0}M$, we can write $h=m_1am_2$ with $m_1,m_2\in M$ and $a\in \widetilde{A}^{+}\cup \widetilde{A}^{+}\mat{0&-1\\1&0}$ in a unique way. \par

Now we show that the generalized Iwasawa decompositions parametrize $G$ except for a codimension one subvariety.  We first consider $G=NAH$. Let $V$ be the set of all horizontal vectors and vertical vectors in $\text{T}^1(\HH^3)$, where a horizontal (vertical) vector is a vector parallel (perpendicular) to $\CC$ in the Euclidean sense.  Let $G_V=\{g\in G: g(X_1)\in V\}$.  We claim the product map $\rho_1$:
\al{\nonumber
\label{1220}&N\times A\times M\times  \left(\widetilde{A}\cup\widetilde{A}\mat{0&-1\\1&0}\right)\times M\longrightarrow G-G_{V}: \\&\rho_1(n,a,m_1,\widetilde{A},m_2):= nam_1\widetilde{A}m_2
}      
is a homeomorphism. \par

Indeed, we notice first that the map $\pi_2\circ\rho$ on the set $$L^1:=\{e\}\times\{e\}\times M\times \left(\widetilde{A}\cup \widetilde{A}\mat{0&-1\\1&0}\right)\times\{e\}$$ gives an identification of $L^1$ with all non-vertical vectors in the unit normal bundle $\text{N}^1(S)$. For any vector $u\in \text{T}^1(\HH)-V$, we can find unique elements $m_1\in M,\widetilde{a}\in\widetilde{A}$ such that $m_1\widetilde{a}(X_1)$ and $u$ point to the same Euclidean direction.  Next we can find a unique element $a\in A$ such that $am_1\widetilde{a}(X_1)$ and $u$ are based in the same horizontal plane.  After that, we can find a unique element $n\in N$ so that  $nam_1\widetilde{a}(X_1)$ and $u$ are based at the same point.  We observe that the actions of $N,A$ on $\text{T}^1(\HH^3)$ preserve Euclidean directions.  Thus we have $nam_1\widetilde{a}(X_1)=u$.  The group $M$ preserves $X_1$, and acts transitively and faithfully on all vectors in $T_e^1(\HH^3)$ normal to $X_1$, so $M$ can be identified with all orthonormal frames based at $\bd{j}$ with the first reference vector $X_1$.  As a result, choosing a unique $m_2\in M$ for the rightmost factor $M$ on the left hand side of \eqref{1220} , we can take the frame $\{X_1,X_2,X_3\}$ at $e$ to any frame at $\pi (u)$ with the first reference vector $u$, by the action of $nam_1\widetilde{a}m_2$. So the claim is established.   Similarly, we have a decomposition $G=HAN$ induced from the decomposition $G=NAH$ by the inverse map of $G$. This decomposition parametrizes all elements in $G-G_V^{-1}$.  \par

\section{Setup of the problem}\label{setup}
Let $\mathcal P\subset \CC$ be a bounded Apollonian gasket, and $\mathcal C=\mathcal C_{\mathcal P}$ be the collection of all centers from $\mathcal P$.  Let $\mathcal P_t$ be the set of the circles from $\mathcal P$ with curvatures $<e^{-t}$ and $\mathcal C_t$ be the set of centers of $\mathcal P_t$.  \par

Fix an open set $E\subset\CC$ with $E\cap \mathcal P\neq\emptyset$ and $\partial E$ empty or piecewise smooth, and fix a multi-set $\boldsymbol{\Omega}=\prod_{1\leq i\leq k}\Omega_i\subset \mathbb C^k $, where $\Omega_i,1\leq i\leq k$ are bounded open subsets of $\mathbb{C}$ with piecewise smooth boundaries. \par
Let $$\mathcal B_t(\Omega_i,z):=(e^{-t}\Omega_i+z)\cap\mathcal C_t,$$ and $$\mathcal N_t(\Omega_i,z):=\#\mathcal B_t(\Omega_i,z).$$ 

We want to study 
\al{\int_{\CC}\prod_{1\leq i\leq k}\bd{1}\{\mathcal N_t(\Omega_i,z)=r_i\}\chi_E(z)dz.}
and 
\al{\int_{\CC}\prod_{1\leq i\leq k}\mathcal N_t(\Omega_i,z)^{\beta_i}\chi_E(z)dz.}
as $t\rightarrow \infty$.

To proceed, first we choose a Kleinian group $\Gamma< PSL(2,\CC)$ whose limit set $\Lambda(\Gamma)=\mathcal{P}$, such that $\Gamma$ transitively on the circles from $\mathcal{P}$.  The existence of $\Gamma$ can be seen as follows: let $$\Gamma_0=\left\langle PSL(2,\mathbb{Z}), \mat{i&1\\0&-i}\right\rangle.$$  One can check that the limit set of $\Gamma_0$ is the closure of the unbounded Apollonian packing $\mathcal P_0$, determined by three mutually tangent circles $\RR,\RR+\bd{i}, C(\bd{i}/2,1/2)$, and $\Gamma_0$ acts transitively on the circles from $\mathcal P_0$.  Since any Apollonian packing $\mathcal P$ can be mapped to $\mathcal P_0$ by a M\"obius transform, the symmetry group $\Gamma$ of $\mathcal P$ can then be taken as a conjugate of $\Gamma_0$.  \par

Recall that $S$ is the geodesic plane with $\partial S=C(0,1)$, then for any isometry $g\in G$, $g(S)$ is also a geodesic plane, so in the upper half-space model, $g(S)$ is either a hemisphere or a vertical plane.\par

Recall the maps $\bd{q}$ from $G$ to $\overline{\HH^3}$, $\bd{q}_\Re$ from $G$ to $\widehat{\CC}$ defined at \eqref{1244}, \eqref{1245}.  If $\infty\not\in g(\partial S)$, there exists a unique geodesic $l_g$ which traverses $g(S)$ perpendicularly.   Then $\bd{q}(g)$ is the intersection of $l_g$ and $g(S)$, and $\bd{q}_{\Re}(g)$ is the other end point of $l(g)$ besides $\infty$, whence we can see that the definitions for $\bd{q}$ and $\bd{q}_{\Re}$ at $g$ with $\infty\in g(\partial S)$ are continuous extensions.  Therefore, both $\bd{q}$ and $\bd{q}_{\RR}$ are continuous everywhere. \par

Let $\bd{r}=\langle r_1,\dots, r_k\rangle, \boldsymbol{\beta}=\langle \beta_1,\dots,\beta_k\rangle$ be multi-indices, where $r_i\in\ZZ_{\geq 0},\beta_i\in \RR_{\geq 0},1\leq i\leq k$, and at least one component of $\bd{r},\boldsymbol{\beta}$ is nonzero.   Let $\Omega_i^*\subset \HH^3$ be the ``chimney"
$$\Omega_i^*:=\{z+r\bd{j}:z\in\Omega_i, r>1\},$$
for $1\leq i\leq k$.  \par
Let $\Gamma_H=\Gamma\cap H$.  Since $Stab(C(0,1))=H$ and $\Gamma$ acts transitively on the circles from $\mathcal P$, we have 
$$\mathcal C=\{\Re(\bd{q}(\gamma)): \gamma\in \Gamma/\Gamma_H  \},$$
and
$$\mathcal C_t=\{\Re(\bd{q}(\gamma)): \gamma\in \Gamma/\Gamma_H, \Im(\bd{q}(\gamma ))>e^{-t}   \}.$$
Therefore, we can rewrite $\mathcal N_t(\Omega_i,z)$ as 
\al{\label{0428}\nonumber
\mathcal N_t(\Omega_i,z)=&\#(e^{-t}\Omega_i+z)\cap \mathcal C_t\\
\nonumber=& \#\{\gamma\in\Gamma/\Gamma_H: \Re(\bd{q}(\gamma))\in e^{-t}\Omega_i+z, \Im(\bd{q}(\gamma))>e^{-t}\}\\
=&\#\{\gamma\in\Gamma/\Gamma_H:  \Re(a_{-t}n_{-z}\bd{q}(\gamma))\in \Omega_i, \Im(a_{-t}n_{-z}\bd{q}(\gamma))>1\}\nonumber\\
=&\#\{\gamma\in\Gamma/\Gamma_H: \bd{q}(a_{-t}n_{-z}\gamma)\in\Omega_i^*  \}.
}

Recall the definitions for the functions $F_{{\boldsymbol{\Omega}},{\bd{r}}}, F_{\boldsymbol{\Omega}}^{\boldsymbol{\beta}}$ on $G$ defined by \eqref{03581} and \eqref{03591}:
\aln{
&F_{{\boldsymbol{\Omega}},{\bd{r}}}(g):=\prod_{1\leq i\leq k }\bd{1}\left\{\#( \bd{q}(g^{-1}\Gamma/\Gamma_H)\cap \Omega_i^*)=r_i  \right\},\\
&F_{\boldsymbol{\Omega}}^{\boldsymbol{\beta}}(g):=\prod_{1\leq i\leq k }\#( \bd{q}(g^{-1}\Gamma/\Gamma_H)\cap \Omega_i^*)^{\beta_i}.
}

Collecting \eqref{0428},\eqref{03581},\eqref{03591}, we have
\al{\label{0438}
&\int_{\CC}\prod_{1\leq i\leq k}\bd{1}\{\mathcal N_t(\Omega_i,z)=r_i\}\chi_E(z)dz=\int_\CC F_{\bd{\Omega},\bd{r}}(n_{z}a_t)\chi_E(z)dz,\\
\label{0439}&\int_{\CC}\left(\prod_{1\leq i\leq k}\mathcal N_t(\Omega_i,z)^{\beta_i}\right)\chi_E(z)dz=\int_\CC F_{\boldsymbol{\Omega}}^{\boldsymbol{\beta}}(n_{z}a_t)\chi_E(z)dz.
}\\

At this point, we have rephrased our problem in the setting of Theorem \ref{03211}.  We restate it here:
\begin{thm}[Mohammadi-Oh, \cite{MO15}]\label{0321} Suppose $\Gamma<G$ is geometrically finite.  Suppose $\Gamma\backslash \Gamma N$ is closed in $\Gamma\backslash G$ and $|\mu_{N}^{\text{PS}}|<\infty$.  For any $\Psi\in C_c^{\infty}(\Gamma\backslash G)$ and any $f\in C^{\infty}(\Gamma\backslash \Gamma{N})$, we have
\al{\label{1258}\lim_{t\rightarrow\infty}e^{(2-\delta)t}\int_{\Gamma\backslash\Gamma N}\Psi(na_t)f(n)d\mu_{N}^{\normalfont \text{Leb}}(n)=\frac{m^{\rm{BR}}(\Psi)\mu_{N}^{\rm{PS}}(f)}{m^{\rm{BMS}}(\Gamma\backslash G) }.}
\end{thm}

Here $m^{\br},\mu_N^{\ps},m^{\bms}$ are certain conformal measures for which we are going into detail in the next section.  In our situation, $\Gamma$ is the symmetry group of the Apollonian gasket $\mathcal P$, $f$ is the characteristic funtion $\chi_E$, and $\Psi$ is $F_{\boldsymbol{\Omega}, \bd{r}}$ or $F_{\boldsymbol{\Omega}}^{\boldsymbol{\beta}}$.  We have $\Gamma\bs \Gamma N=N$ as $\Gamma\cap N=\{e\}$.  Since $\Gamma$ is geometrically finite,  we have $0<m^{\bms}(\Gamma\bs G)<\infty$ (Corollary 1.3, \cite{BJ97}).  We will also see that $\mu_{N}^{\rm{PS}}(\chi_E)<\infty$.  The issue for us to apply Theorem \ref{03211} is, none of the functions  $f,F_{\boldsymbol{\Omega}, \bd{r}}$ or $F_{\boldsymbol{\Omega}}^{\boldsymbol{\beta}}$ is continuous.  Moreover, $F_{\boldsymbol{\Omega}, \bd{r}}$, $F_{\boldsymbol{\Omega}}^{\boldsymbol{\beta}}$ are not compactly supported, so apriori $m^{\br}(F_{\boldsymbol{\Omega}, \bd{r}})$, $m^{\br}(F_{\boldsymbol{\Omega}}^{\boldsymbol{\beta}})$ can be $\infty$. The purpose of the next two sections is to prove Theorem \ref{07141}, which is an extended version of Theorem \ref{03211}. Along the way we will show that  $m^{\br}(F_{\boldsymbol{\Omega},\bd{r}}), m^{\br}(F_{\boldsymbol{\Omega}}^{\boldsymbol{\beta}})<\infty$.\par

\section{Conformal Measures}\label{0129}
We keep all notation from previous sections.  Let $\Gamma< G$ be a discrete group with the limit set $\Lambda(\Gamma)=\overline{\mathcal{P}}$ and acting transitively on the circles from $\mathcal{P}$.  A family of finite measures $\{\mu_x:x\in\mathbb{H}^3\}$ on $\partial \mathbb H^3$ is called a \emph{$\Gamma$-invariant conformal density} of dimension $\delta_\mu>0$ if for any $x,y\in\mathbb{H}^3, u\in\partial\mathbb{H}^3$, 
\aln{
\gamma^*\mu_x=\mu_{\gamma x},\hspace{0.2in}\text{and} \hspace{0.2in}\frac{d\mu_x(u)}{d\mu_y(u)}=e^{-\beta_u(x,y)\delta_{\mu}},
}
where for any Borel set $F\subset\partial \mathbb{H}^n$, $\gamma^{*}\mu_x(F)=\mu_x(\gamma^{-1}F)$.  The function $\beta_u$ is the Busemann function defined as: $$\beta_u(x,y)=\lim_{t\rightarrow\infty}d(u_t, x)-d(u_t,y),$$ where $u_t$ is any geodesic ray tending to $u$ as $t\rightarrow \infty$.\par

Two particularly important densities are the Lebesgue density $\{m_x:x\in\HH^3\}$ and the Patterson-Sullivan density $\{\nu_x:x\in\HH^3\}$.  The Lebesgue density is a $G$-invariant density of dimension 2, and for each $x$, $m_{x}$ is $Stab(x)$-invariant.  The Patterson-Sullivan density $\{\nu_x\}$ is supported on the limit set $\overline{\mathcal P}$, and of dimension $\delta$ \cite{Su79}.  Both densities are unique up to scaling. We normalize these densities so that $|\nu_{\bd{j}}|=1$ and $|m_{\bd{j}}|=\pi$.\par

Write $z=x+y\bd{i}$.  We have an explicit formula for $m_{\bd{j}}$ in the $\CC$ coordinate:  
\al{\label{0713}dm_{\bd{j}}(z)=\frac{dxdy}{(1+x^2+y^2)^2}.}
Therefore, $dm_{\jj}(z)\approx dxdy$ near 0. 

The formula for $\nu_{\bd{j}}$ is explicitly given as the weak limit as $s\rightarrow \delta^{+}$ of the family of measures 
$$\nu_{\jj,s}:=\frac{1}{\sum_{\gamma\in\Gamma}e^{-sd(\jj,\gamma\jj)}}\sum_{\gamma\in\Gamma}e^{-sd(\jj,\gamma\jj)}\delta_{\gamma\jj},$$
where $\delta_{\gamma\jj}$ is the Dirac delta measure supported at the point $\gamma\jj$.\par 

We have the following estimate for $\nu_{\bd{j}}(B(z,r))$, where $B(z,r)\subset\CC$ is the Euclidean ball centered at $z$ with radius $r$ (see Sec. 7 of \cite{Su84}): 
\al{\label{2232}\nu_{\bd{j}}(B(z,r))\ll \min\{r^{\delta},1\}.} 

By a simple packing argument, \eqref{2232} implies $\nu_{\bd{j}}(l)=0$ for any differentiable curve $l\subset\CC$. So by our assumption for $E$, we have $\nu_{\bd{j}}(\partial E)=0$.

We also need to work with certain measures related to the conformal densities $\{m_x: x\in\mathbb H^3\}$ and $\{\nu_x: x\in\mathbb H^3\}$.  For any $u\in \text{T}^1(\mathbb H^3)$, let $u^{-},u^+\in \widehat{\mathbb C}$ be the starting and ending points of $u$. We can identify $N$ with $\partial \HH^3-\{0\}$ via the map $g\rightarrow g(X_0)^+$.  Let $H_1=H/M$, then $H_1$ can be identified with $\partial \HH^3-\partial S$ via the map $g\rightarrow g(X_0)^-$.  We define measures $\mu_N^{\ps},\mu_{H_1}^{\ps}$ as: 
\al{
\label{0811}
&d\mu_{N}^{\text{Leb}}(n):=e^{2\beta_{n(X_1)^+}(\bd{j}, n(\bd{j}))}dm_{\bd{j}}(n(X_1)^+),\\
\label{0812}
&d\mu_{N}^{\text{PS}}(n):=e^{\delta\beta_{n(X_1)^+}(\bd{j}, n(\bd{j}))}d\nu_{\bd{j}}(n(X_1)^+),\\
\label{0813}
&d\mu_{H_1}^{\ps}(h_1):=e^{\delta\beta_{h_1(X_1)^-}(\bd{j}, h_1(\bd{j}))}d\nu_{\bd{j}}(h_1(X_1)^-).
}
Later on it will follow from Lemma \ref{0741} that $\mu_N^{\text{Leb}}(n_z)=dz$, so $\mu_N^{\text{Leb}}$ is in fact a Haar measure on $N$. \par

We can lift the measure $\mu_{H_1}^{\ps}$ to a unique right $M$-invariant measure $\mu_H^{\ps}$ on $H$ satisfying:
for any $f\in C_c(H_1)$, define $\hat{f}\in C_c(H)$ as $$\hat{f}(h)=f(hM).$$ Then 
$$\int_{H}\hat{f}(h)d\mu_{H}^{\ps}(h)=\int_{H_1}f(h_1)d\mu_{H_1}^{\ps}(h_1).$$ 
We can view $H$ as a circle bundle over $H_1$. Under this viewpoint, from the definition of $\mu_{H}^{\ps}$ we have 
$$\mu_{H}^{\ps}=d\mu_{H_1}^{\ps}\cdot dm_M^{\text{Haar}},$$
where $m_M^{\text{Haar}}$ is the Haar measure of $M$ with $|m_M^{\text{Haar}}|=1$.

\subsection{Finiteness of $\mu_N^{\normalfont\ps}(N), \mu_{H_1}^{\normalfont\ps}(\Gamma_H\bs H_1)$ and $ \mu_{H}^{\normalfont\ps}(\Gamma_H\bs H)$}  In this section we are going to show $0<\mu_N^{\normalfont\ps}(N), \mu_{H_1}^{\normalfont\ps}(\Gamma_H\bs H_1), \mu_{H}^{\normalfont\ps}(\Gamma_H\bs H)<\infty$.  The $>0$ part is trivial and we focus on the $<\infty$ part.  We begin with a calculation:

\begin{lem}\label{0227}For any $q\in\HH^3$, we have $\beta_{\infty}(\jj,q)=\log \Im(q)$.
\end{lem}
\begin{proof}
By the definition of the Buseman function,  
\al{\label{0217}
\beta_{\infty}(\jj, q)=\lim_{t\rightarrow\infty}d(e^{t}\jj,\jj)-d(e^{t}\jj,q)=t-\lim_{t\rightarrow\infty}d(\jj,e^{-t}q).
}

From the hyperbolic Disance formula \eqref{0518},
\al{\nonumber d(\jj,e^{-t}q)=&\text{Arccosh}\left(1+\frac{e^t|\jj-e^{-t}q|^2}{2\Im(q)}\right)\\\nonumber=&\text{Arccosh}\left(1+\frac{e^{t-\log \Im(q)}}{2}(1+O_q(e^{-t}))\right)\\=&t-\log\Im(q)+O_q(e^{-t}).\label{0216}}
Applying \eqref{0216} to \eqref{0217}, we obtain 
$$\beta_{\infty}(\jj, q)=\lim_{t\rightarrow\infty} t- (t-\log\Im(q)+O_q(e^{-t}))=\log \Im(q).$$

\end{proof}

Returning to \eqref{0812}, we have 
\al{
e^{\delta\beta_{n(X_1)^+}(\bd{j}, n(\bd{j}))}=e^{\delta\beta_{n(0)}(\bd{j}, n(\bd{j}))}=e^{-\delta\beta_{0}( \bd{j},n^{-1}(\bd{j}))}=  e^{-\delta\beta_{\infty}\left( \bd{j},\mat{0&-1\\1&0}n^{-1}\left(\bd{j}\right)\right)}=(|n^{-1}(0)|^2+1)^{\delta}
}
Since $\Lambda(\Gamma)=\overline{\mathcal P}$ is compact, the term $(|n^{-1}(0)|^2+1)^{\delta}$ is bounded on the support of $\Lambda(\Gamma)=\overline{\mathcal P}$ of $\mu_N^{\ps}$. As $|v_{\jj}|$ is finite, we have $\mu_N^{\ps}(N)<\infty$.\par

Now we consider $\mu_H^{\ps}(\Gamma_H\bs H_1)$ and $\mu_H^{\ps}(\Gamma_H\bs H)$.  Both $\Gamma_H\bs H$ and $\Gamma\bs G$ have one cusp, whose ranks in both $\Gamma_H$ and $\Gamma$ are equal to 1.   Therefore $\mu_{H_1}^{\ps}$ is compactly supported in $\Gamma_{H} \bs H_1$ from Theorem 6.3 \cite{OS13}.  Thus the term $e^{\delta\beta_{h_1(X_1)^-}(\bd{j}, h_1(\bd{j}))}$ from \eqref{0813} is bounded on the support of $\mu_{H_1}^{\ps}$, so that $\mu_{H}^{\ps}(\Gamma_{H}\bs H)={\mu_{H_1}^{\ps}(\Gamma_{H}\bs H_1)}<\infty$.

\subsection{Quasi-product Conformal Measures on $\text{T}^1(\mathbb H^3)$ and $G$}
Following Roblin \cite{Ro00}, given two conformal measures $\{\mu_x\},\{\mu_x'\}$, we can define a quasi-product measure $\widetilde{m}^{\mu,\mu'}$ on $\text{T}^1(\mathbb H^3)$  by 
\aln{d\widetilde{m}^{\mu,\mu'}(u)=e^{\delta_\mu\beta_{u^+}(o,\pi(u))}e^{\delta_{\mu'}\beta_{\mu^-}(o,\pi(u))}d\mu_o(u^{+})d\mu_o'(u^-)ds,
} 
where $o$ is any point in $\mathbb H^3$, $u\in \text{T}^1(\mathbb H^3)$, $u^-,u^+\in\partial \mathbb H^3$ are the starting and ending points of the geodesic ray containing $u$, and $s=\beta_{u^-}(o,\pi(u))$.  It is an exercise to check that
\begin{enumerate}[i)]
 \item The definition of $\widetilde{m}^{\mu,\mu'}$ is independent of the chosen base point $o$.
 \item The measure $\widetilde{m}^{\mu,\mu'}$ is left $\Gamma$-invariant.
\end{enumerate}

We can lift the measure $m^{\mu,\mu'}$ to a unique right $M$-invariant measure on $G$ satisfying:
for any $f\in C_c(\text{T}^1(\HH^3))$, define $\hat{f}\in C_c(G)$ as $$\hat{f}(g)=f(g(X_1)).$$ Then 
$$\int_{G}\hat{f}(g)dm^{\mu,\mu'}(g)=\int_{T^1(\HH^3)}{f}(u)d\widetilde{m}^{\mu,\mu'}(u).$$ 
We can view $G$ as a circle bundle over $\text{T}^1(\HH^3)$, and the right action of $M$ on $G$ preserves fibers.  From the right $M$-invariance of $m^{\mu,\mu'}$, we have 
$$dm^{\mu,\mu'}=d\widetilde{m}^{\mu,\mu'}\cdot dm_{M}^{\text{Haar}}.$$
 
By the $\Gamma$-invariance, the measures $\widetilde{m}^{\mu,\mu'},m^{\mu,\mu'}$ naturally descend to measures on $\Gamma\backslash\text{T}^1(\HH^3), \Gamma\backslash G$, for which we keep the same notation.  For a left $\Gamma$-invariant function $F$ on $\Gamma\bs G$, we denote the integral $\int_{\Gamma\bs G}F(g)dm^{\mu,\mu'}(g)$ by $m^{\mu,\mu'}(F)$.  \\

We choose the base point $o=\jj$.  The following two quasi-product measures will appear in our analysis:
\begin{enumerate}
\item $\mu=m_{\bd{j}}, \mu'=\nu_{\bd{j}}$; we denote the measure on $\text{T}^1(\HH^3)$ by $\widetilde{m}^{\text{BR}}$, and the measure on $G$ by $m^{\text{BR}}$. These measures are called the Burger-Roblin measures.
\item $\mu=\nu_{\bd{j}}, \mu'=\nu_{\bd{j}}$; we denote the measure on $\text{T}^1(\HH^3)$ by $\widetilde{m}^{\text{BMS}}$, and the measure on $\Gamma\backslash G$ by $m^{\text{BMS}}$. These measures are called the Bowen-Margulis-Sullivan measures.
\end{enumerate}

We point out a few useful properties of these quasi-product measures.  \par
The Burger-Roblin measures and the Bowen-Margulis-Sullivan measures are locally finite and regular Borel measures, which vanish on a countable union of submanifolds of $\text{T}^{1}(\HH^3)$ or $G$ of codimension $\geq 1$ (for instance, algebraic subvarieties of $G$ of codimension $\geq 1$).  This is because locally, the Burger-Roblin measures and the Bowen-Margulis-Sullivan measures are products of measures ($m_{\bd{j}},\nu_{\bd{j}},m_{\RR}^{\text{Harr}},m_{M}^{\text{Haar}}$), each of which is locally finite, regular and vanishes on submanifolds of codimension $\geq1$ of its corresponding measure space. \par

Finally, we have $0<m^{\bms}(\Gamma\bs G)<\infty$, which follows from the geometrically finiteness of $\Gamma$ (see Page 270 of \cite{Su84}).

\subsection{Computation of ${m}^{\normalfont\text{BR}}$ in the generalized Iwasawa Coordinates} 
The purpose of this section is to compute $m^{\text{BR}}$ in the $HAN$ coordinates (Proposition \ref{0130}).  We further write $H$ into its Cartan decomposition \eqref{0532}.  This decomposition provides an explicit fibration of $H$ over $H_1$, with the first two factors $M\times\left(\widetilde A^+\cup\widetilde A^+\mat{0&-1\\1&0}\right)$ of  \eqref{0532} parametrize $H_1=H/M$ except for two points ${M},{\mat{0&-1\\1&0}}M$.  For this reason and for simplicity we abuse notation, writing $$H_1=M\times\left(\widetilde A^+\cup\widetilde A^+\mat{0&-1\\1&0}\right),$$
ignoring the two points $M$ and $\mat{0&-1\\1&0}M$.\par
We first observe that the product map $\rho_2$:
\aln{& L_2:=M\times\left( \widetilde{A}^{+}\cup \widetilde{A}^{+}\mat{0&-1\\1&0}\right) \times A\times N\rightarrow \text{T}^1(\HH^3),\\
&\rho_2(m,\tilde a,a, n):=m\tilde a an(X_1)
}
embeds $L_2$ into an Zariski-open subset of $\text{T}_1(\HH^3)$, by a consideration similar to an earlier one for the $NAH$ decomposition below \eqref{1220}.  Under the $H_1 AN$ coordinates of $\text{T}^1(\HH^3)$, we can compute \par

\al{\nonumber
d\widetilde{m}^{BR}(h_1a_{t_1}n_z(X_1))=&e^{2\beta_{h_1a_{t_1}n_zX_1^+}(\jj,h_1a_{t_1}n_z\jj)}\cdot e^{\delta\beta_{h_1a_{t_1}n_zX_1^-}(\jj, h_1a_{t_1}n_z\jj)}\\
\nonumber &dm_{\jj}(h_1a_{t_1}n_zX_1^+)d\nu_{\jj}(h_1a_{t_1}n_zX_1^-)dt\\
\nonumber =& e^{2\beta_{h_1a_{t_1}n_z0}(\jj,h_1a_{t_1}n_z\jj)}\cdot e^{\delta\beta_{h_1\infty}(\jj,h_1a_{t_1}n_z\jj)}\\
&dm_{\jj}(h_1a_{t_1}n_z0)d\nu_{\jj}(h_1\infty)dt, \label{0451}
}
where $t=\beta_{h_1\infty}(\jj,h_1a_{t_1}n_z\jj)$.   \par

Applying Lemma \ref{0227} to $t=\beta_{h_1\infty}(\jj,h_1a_{t_1}n_z\jj)$, we obtain
\al{\label{0452}
t=\beta_{h_1\infty}(\jj,h_1a_{t_1}n_z\jj)=&\beta_{\infty}(h_1^{-1}\jj,a_{t_1}n_z\jj)=\beta_{\infty}(h_1^{-1}\jj,\jj)+\beta_{\infty}(\jj, a_{t_1}n_z\jj)\nonumber\\
=&-\log(h_1^{-1}\jj)-t_1.
}
Combining \eqref{0451} and \eqref{0452}, we obtain 
\al{\label{1649}
d\widetilde{m}^{BR}(h_1a_{t_1}n_z(X_1))=\frac{1}{\Im (h_1^{-t}\bd{j})^{\delta}}e^{2\beta_{h_1a_{t_1}n_z0}(\jj,h_1a_{t_1}n_z\jj)}e^{-\delta t_1}dm_{\jj}(h_1a_{t_1}n_z0)d\nu_{\jj}(h_1\infty)dt_1.
}

\begin{lem} \label{0450}For any $g\in G$, consider the measure $\lambda_g$ on $N$ given by 
\al{\label{0741}d\lambda_g(n_z)=e^{2\beta_{gn_z0}(\jj,gn_z\jj)}dm_{\jj}(gn_z0).}
Then $\lambda_g=\lambda_e$ and $\lambda_e$ is a Haar measure on $N$.
\end{lem}
\begin{proof}
By the $G$-invariance of $\{m_x\}$,
\al{\nonumber
dm_{\jj}(gn_z0)&=dm_{g^{-1}\jj}(n_z0)=e^{2\beta_{n_z0}(\jj,g^{-1}\jj)} dm_{\jj}(n_z0)\\
&=e^{2\beta_{n_z0}(\jj,g^{-1}\jj)}dm_{\jj}(n_z0).
}
Therefore, 
\al{
 d\lambda_g(n_z)= e^{2\beta_{n_z0}(g^{-1}\jj,n_z\jj)}\cdot  e^{2\beta_{n_z0}(\jj,g^{-1}\jj)}dm_{\jj}(n_z0)=d\lambda_e(n_z).\label{0745}
}
Combining \eqref{0741},\eqref{0745} and \eqref{0713}, we have $d\lambda_e(n_z)=d\lambda_{n_z}(e)=d\lambda_e(e)=dz$, so $\lambda_e$ is a Haar measure on $N$.
\end{proof}

Recall the definition \eqref{0813} for $\mu_{H_1}^{\text{PS}}$.  We can use Lemma \ref{0227} to compute
\al{\label{0506}d\mu_{H_1}^{\text{PS}}(h_1)=\frac{1}{\Im(h_1^{-1}\jj)^{\delta}}d\nu_{\jj}(h_1\infty)}

Collecting \eqref{1649}, Lemma \ref{0450} and \eqref{0506}, we obtain
$$d\widetilde{m}^{\text{BR}}(h_1a_{t_1}n_{z}X_{1})=e^{-\delta t_1}d\mu_{H_1}^{\text{PS}}(h_1)dzdt_1.$$
Therefore, in the $H_1ANM$ decomposition for $G$, for any $h_1\in  H_1,a_{t_1}\in A, n_z\in N, m\in M$, we have 
$$d{m}^{\text{BR}}(h_1a_{t_1}n_{z}m)=e^{-\delta t_1}d\mu_{H_1}^{\text{PS}}(h_1)dzdt_1dm_M^{\text{Haar}}(m),$$
by the right $M$-invariance of $m^{\br}$.\par

The decompositions $H_1ANM$ and $H_1MAN$ are related as follows: 
If $h_1a_{t_1}n_zm=h_1'm'a_{t_1'}n_{z'}$, then $h_1'=h_1, t_1'=t_1, m'=m, z'=m^{-1}z$. Therefore, in the $H_1MAN$ decomposition, let $h_1\in H_1, m\in M, a_{t}\in A, n_{z}\in N$, then the Burger-Roblin measure $m^{\text{BR}}$ is given by 
$$d{m}^{\text{BR}}(h_1ma_{t}n_{z})=e^{-\delta t}d\mu_{H_1}^{\rm {PS}}(h_1)dzdt_1dm.$$

Write $h=h_1m$. Since $$d\mu_H^{\rm{PS}}(h_1m)=d\mu_{H_1}^{\rm{PS}}(h_1)\cdot dm,$$ 
we obtain
\begin{prop}\label{0130} In the $HAN$ decomposition, let $h\in H, a_{t}\in A, n_{z}\in N$.  Then the Burger-Roblin measure $m^{\text{BR}}$ is given by 
$$d{m}^{\text{BR}}(ha_{t}n_{z})=e^{-\delta t}d\mu_{H}^{\rm {PS}}(h)dzdt.$$
\end{prop}

\section{Equidisribution of Expanding Horospheres}\label{equidstribution}
 
The purpose of this section is to prove Theorem \ref{07141}, which is an extension of Theorem \ref{03211}.  \par
Let $\mathcal W$ be the set of pairs $(f,\Psi)$ satisfying:
\begin{equation}
\mathcal W:=\left\{(f,\Psi):\begin{array}{@{}ll@{}}f\in L^1(\Gamma\backslash \Gamma N, \mu_{N}^{\text{PS}}), \Psi\in L^1(\Gamma\backslash G,m^{\text{BR}}), \\ \lim_{t\rightarrow\infty}e^{(2-\delta)t}\int_{\Gamma\backslash\Gamma N}\Psi(na_{t})f(n)d\mu_N^{\text{Leb}}(n)=\frac{m^{\rm{BR}}(\Psi)\mu_{N}^{\rm{PS}}(f)}{m^{\rm{BMS}}(\Gamma\backslash G)}\end{array}\right \}.
\end{equation}
We first observe that $\mathcal W$ inherit some linear structure:
\begin{enumerate}[(i)]
\item If $(f,\Psi)\in\mathcal W$, then for any $\alpha_1,\alpha_2\in\CC$, $(\alpha_1 f, \alpha_2 \Psi)\in\mathcal W$.
\item If $(f_1,\Psi), (f_2,\Psi)\in\mathcal W$, then  $(f_1+f_2, \Psi)\in\mathcal W$.
\item If $(f,\Psi_1), (f,\Psi_2)\in\mathcal W$, then $(f,\Psi_1+\Psi_2)\in \mathcal W$.
\end{enumerate}

The smoothness assumption in Theorem \ref{03211} is for obtaining an effective convergence rate. This is not needed for our purpose here. 

By the same method from \cite{OS12}, one can extend Theorem \ref{03211} to $\Psi\in C_c(\Gamma\backslash G)$ and $f\in L^1(\Gamma\backslash \Gamma N, \mu_{N}^{\text{PS}})$ with $\lim_{\epsilon\rightarrow 0}\mu_{N}^{\text{PS}}(f_{\epsilon^+}-f_{\epsilon^-})=0$, where
\al{\label{1000}
f_{\epsilon^+}(n_z):=\sup_{|w-z|<\epsilon}f(n_w),\\
f_{\epsilon^-}(n_z):=\inf_{|w-z|<\epsilon}f(n_w).\label{0545}
} \par

However, this is still not enough for our purpose.  We need to extend Theorem \ref{03211} to cover some nonnegative functions $f$ and $\Psi$, with $\Psi\in L^1(\Gamma\backslash G, m^{\text{BR}})$ and non-compactly supported.  Indeed, in the lattice case, the measure $m^{\text{BMS}}$ on $\Gamma\bs G$ is just the Haar measure, and Shah obtained Theorem \ref{03211} for any $f\in L^1(\Gamma\backslash\Gamma N, m_N^{\text{Haar}})$ and $\Psi\in L^1(\Gamma\backslash G, m_G^{\text{Haar}})$ \cite{Sh91}. However, it seems that removing the compactly supported assumption for $\Psi$ in the infinite co-volume situation is a much more delicate issue.  In fact in the works \cite{KO11}, \cite{OS13}, \cite{MO15}, which deal with the infinite co-volume situation, the compactly supported assumption seems crucially used in proving the equidisribution theorems of expanding horospheres.  To see one subtlety here, compared to the lattice case, in the statement of Theorem \ref{03211}, we have an extra factor $e^{(2-\delta)t}$, which goes to infinity as $t$ does.  We haven't been able to fully extend  Theorem \ref{03211} to cover $\Psi\in L^1(\Gamma\bs G, m^{BR})$, and we circumvent this difficulty by observing some hierarchy structure in the set $\mathcal W$ (Proposition \ref{1426}), which is enough for our purpose.  \par

In Section \ref{1315} we prove the membership of certain pairs in $\mathcal W$ using Theorem \ref{0827}, in Section \ref{1317} we prove some hierarchy structure in $\mathcal{W}$, and in Section \ref{1320} we finish the proof of Theorem \ref{07141}.  

\subsection{Membership of certain pairs in $\mathcal W$\label{1315}}
Let $E\subset\CC$ be an open set with $\partial E$ empty or piecewise smooth, and let $\Omega\subset\CC$ be a bounded open set with $\partial \Omega$ piecewise smooth.   
First we claim that $(f_0, \Psi_0)\in\mathcal W$, where 
\al{\label{2241}&f_0(n_z):=\chi_E(z),\\
&\label{2242}\Psi_0(g):=\sum_{\gamma\in \Gamma/\Gamma_H}\bd{1}\{\bd{q}(g^{-1}\gamma)\in {\Omega^*}   \},}
recalling that $\Omega^*$ is the infinite chimney based at $\Omega$ (see \eqref{0530}). 
We will see shortly that the pair $(f_0, \Psi_0)$ is related to counting circles in $E$.\par

We first calculate the right hand side of \eqref{1258} with $f=f_0$ and $\Psi=\Psi_0$.

Write $g=h_1ma_t n_z$ in the $H_1MAN$ coordinate.  From Proposition \ref{0130}, 
\al{\nonumber
m^{\text{BR}}(\Psi_0)=&\int_{g\in\Gamma\backslash G}\sum_{\gamma\in\Gamma_H\backslash\Gamma}\bd{1}\{\bd{q}((\gamma g)^{-1})\in\Omega^*\}d m^{\text{BR}}(g) \\
\nonumber
=&\int_{g\in\Gamma_H\backslash G}\bd{1}\{\bd{q}( g^{-1})\in\Omega^*\}dm^{\text{BR}}(g)\\
\nonumber
=&\int_{h\in\Gamma_H\backslash H} \int_{t\in\RR}\int_{z\in\CC} \bd{1}\{\bd{q}( (ha_tn_z)^{-1})\in\Omega^*\}\cdot e^{-\delta t}dzdtd\mu_{H}^{\rm {PS}}(h)\\
\nonumber
=&\int_{h\in\Gamma_H\backslash H}\int_{t>0}\int_{z\in-\Omega}  e^{-\delta t} dzdtd\mu_{H}^{\rm {PS}}(h)\\
=&\frac{1}{\delta}\cdot \mathcal Area(\Omega)\mu_{H}^{\text{PS}}(\Gamma_H\backslash H).
}

Next, we have $\mu_{N}^{\text{PS}}(f_0)=w(E)$, recalling that the measure $w$ on $\CC$ is the pull back measure of $\mu_{N}^{\text{PS}}$ under the map $z\rightarrow n_z$.  We also have $m^{BMS}(\Gamma\backslash G)= \widetilde{m}^{BMS}(\Gamma\backslash \HH^3)$ and $\mu_{H}^{\ps}(\Gamma\bs H)=\mu_{H_1}^{\ps}(\Gamma\bs H_1)$.  Therefore, 
\al{\label{1750}
\frac{m^{\rm{BR}}(\Psi_0)\mu_{N}^{\rm{PS}}(f_0)}{m^{\rm{BMS}}(\Gamma\backslash G)}=\frac{\mathcal Area(\Omega)\mu_{H_1}^{\text{PS}}(\Gamma_H\backslash H_1)w(E)}{\delta \cdot \widetilde{m}^{\text{BMS}}(\Gamma\backslash \HH^3)}.
}

We now turn to the left hand side of \eqref{1258}.  Recall that $\Gamma\bs\Gamma N=N$ as $\Gamma\cap N=\{e\}$.  We have 
\al{\label{0920}\nonumber
&e^{(2-\delta)t}\int_{N} f_0(n_z)\Psi _0(n_za_t)dz\\
\nonumber=& e^{(2-\delta)t}\int_{ N} \chi_E(z)\sum_{\gamma\in\Gamma/\Gamma_H}\bd{1}\{\bd{q}(a_{-t}n_{-z}\gamma)\in \Omega^* \}dz\\
=&e^{(2-\delta)t}\int_{ N} \chi_E(z)\sum_{\gamma\in\Gamma/\Gamma_H}\bd{1}\{ z_\gamma-z\in e^{-t}\Omega; r_\gamma>e^{-t} \}dz.
}
where we wrote $\bd{q}(\gamma)=z_\gamma+r_\gamma \bd{j}$. \par
Let $N(E, t):=\#\mathcal C_t\cap E$, and denote the diameter of $\Omega$ by $\mathcal D(\Omega)$.  For any $\epsilon>0$, we let 
\aln{
&E_{\epsilon^+}:=\{x\in\CC: d(x,E)<\epsilon   \},\\
&E_{\epsilon^-}:=\{x\in E: d(x,\partial E)>\epsilon\}.
}

We have 
\al{\label{0115}
e^{-\delta t}\mathcal Area(\Omega)\cdot N(E_{(e^{-t}\mathcal D(\Omega))^-},t)\leq \eqref{0920}\leq e^{-\delta t}\mathcal Area(\Omega)\cdot N(E_{(e^{-t}\mathcal D(\Omega))^+},t).
}

The quantity $N(*,t)$ can be estimated via the following more detailed version of Theorem \ref{0827}:
\begin{thm}[Oh-Shah, Theorem 1.6, \cite{OS12}]\label{circlecounting} Let $\mathcal{P}$ be a bounded Apollonian circle packing.  Let $E\subset \CC$ be an open set with no boundary or piecewise smooth boundary.  Then 
 $$\lim_{t\rightarrow\infty}\frac{N(E,t)}{e^{\delta t}}=\frac{\mu_{H_1}^{\text{PS}}(\Gamma_H\backslash H_1)w(E)}{\delta \cdot m^{\rm{BMS}}(\Gamma\backslash \HH^3)}.$$
\end{thm}
So comparing Theorem \ref{0827} and Theorem \ref{circlecounting}, we can see $v(E)$ and $w(E)$ are off by a constant factor:

$$v(E)=\frac{\mu_{H_1}^{\text{PS}}(\Gamma_H\backslash H_1)}{\delta \cdot m^{\rm{BMS}}(\Gamma\backslash \HH^3)}w(E).$$

Applying Theorem \ref{circlecounting} to \eqref{0115} with $E$ replaced by $E_{\epsilon^\pm}$, we have
\al{\label{1251} \lim_{t\rightarrow\infty}\frac{N(E_{\epsilon^\pm},t)}{e^{\delta t}}=\frac{\mu_{H_1}^{\text{PS}}(\Gamma_H\backslash H_1)w(E_{\epsilon^\pm})}{\delta \cdot m^{\text{BMS}}(\Gamma\backslash \HH^3)}.
}

Noting that $\lim_{\epsilon\rightarrow0}w(E_{\epsilon^+})-w(E_{\epsilon^-})=0$ as $\partial E$ is piecewise smooth, and letting $t$ goes to infinity for \eqref{0115}, we obtain
\al{\label{0120}\lim_{t\rightarrow\infty}e^{(2-\delta)t}\int_{N} f_0(n_z)\Psi _0(n_za_t)dz=\frac{\mathcal Area(\Omega)\mu_{H_1}^{\text{PS}}(\Gamma_H\backslash H_1)w(E)}{\delta \cdot m^{\text{BMS}}(\Gamma\backslash \HH^3)},}
which agrees with \eqref{1750}.\par

\subsection{The hierarchy structure in $\mathcal W$\label{1317}} For any $\Psi\in L^1(\Gamma\backslash G)$, let $\mathcal Supp(\Psi)$ be the support of $\Psi$ and $\mathcal Disc(\Psi)$ be the set of discontinuities of $\Psi$.  We aim to prove the following proposition. \par

\begin{prop}\label{1426} Suppose $f\in L^1(\Gamma\bs \Gamma N,\mu_N^{\ps})$, nonnegative, and $\lim_{\epsilon\rightarrow 0}\mu_{N}^{\text{PS}}(f_{\epsilon^+}-f_{\epsilon^-})=0$.  Suppose $\Psi\in L^1(\Gamma\bs G, m^{\rm {BR}})$, nonnegative, $\Vert\Psi\Vert_{L^\infty}<\infty$, and $m^{\rm {BR}}\left(\overline{\mathcal Disc (\Psi)}\right)=0$.  If $(f,\Psi)\in \mathcal W$, then for any Borel measurable function $\widetilde{\Psi}$ with $0\leq \widetilde{\Psi}\leq \Psi$ and $m^{\rm {BR}}\left(\overline{\mathcal Disc (\widetilde{\Psi})}\right)=0$, we have $(f,\widetilde{\Psi})\in \mathcal W$.
\end{prop}
\begin{proof}
First we prove the following claim.\par
\noindent$\bd{Claim}$: for an $\epsilon>0$, there exits $\Psi_\epsilon \in C_c(\Gamma\backslash G, m^{\text{BR}})$ such that $0\leq\Psi_\epsilon\leq\Psi$, $m^{\text{BR}}(\Psi-\Psi_\epsilon)<\epsilon$ and $\Psi_\epsilon$ is supported away from the discontinuities of $\Psi$ and $\widetilde{\Psi}$. \par
Since $\Gamma\backslash G$ is second countable and $m^{\text{BR}}$ is a regular Borel measure on $\Gamma\bs G$, we can find a compact set $K_{\epsilon}\subset \Gamma\backslash G$ such that 
$$\int_{\Gamma\bs G-K_{\epsilon}}\Psi(g)dm^{\br}(g)<\epsilon/2.$$
We also choose a relatively compact open set $V_\epsilon\subset\Gamma\bs G$ such that $K_\epsilon\subset V_\epsilon$.\par 

Since $m^{\br}$ is a regular Borel measure on $\Gamma\bs G$ and $$m^{\br}\left(\overline{\mathcal Disc(\Psi)} \cup \overline{\mathcal Disc(\widetilde{\Psi})}\right)=0,$$ we can find two open sets $U_\epsilon,U_\epsilon'\subset \Gamma\bs G$ such that 
$$\overline{\mathcal Disc(\Psi)} \cup \overline{\mathcal Disc(\widetilde{\Psi})}\subset U_\epsilon\subset \overline{U_\epsilon}\subset U_\epsilon' $$
and
$$m^{\br}(U_\epsilon')<\frac{\epsilon}{2\max\{1,\Vert\Psi\Vert_{L^\infty}\}}.$$
From the Tietze Extension Theorem, there exists a function $\Phi_{\epsilon}\subset C(\Gamma\bs G)$ such that $0\leq \Phi_\epsilon\leq 1$, $\Phi_\epsilon \equiv 1 $ on $K_{\epsilon}-U_\epsilon'$ and $\Phi_\epsilon\equiv 0$ on $\overline{U_\epsilon}\cup (\Gamma\bs G-V_\epsilon)$. \par
Now set $\Psi_\epsilon=\Psi\cdot\Phi_\epsilon$, then we can see that $\Psi_\epsilon$ is compactly supported as $\Phi_\epsilon$ is, $\Psi_\epsilon$ is continuous as ${\mathcal Supp}(\Psi_\epsilon)\cap \mathcal Disc(\Psi)=\emptyset$, and $0\leq \Psi_\epsilon\leq \Psi$.  Therefore, 
\al{\nonumber\int_{\Gamma\bs G}\Psi(g)-\Psi_\epsilon(g)dm^{\text{BR}}(g)&\leq  \int_{\Gamma\bs G-K_{\epsilon}}\Psi(g)-\Psi_\epsilon(g)dm^{\text{BR}}(g)+ \int_{ U_\epsilon'}\Psi(g)-\Psi_\epsilon(g)dm^{\text{BR}}(g)  \\
&< \epsilon/2+\epsilon/2=\epsilon,
}
finishing the proof of the claim.  \par

Next, according to the comment around \eqref{1000}, for each $\epsilon$, $(f,\Psi_\epsilon)\in\mathcal W$.  Therefore, $(f,\Psi-\Psi_\epsilon)\in\mathcal W$, so that
$$\lim_{t\rightarrow\infty}e^{(2-\delta)t}\int_{\Gamma\bs\Gamma N}f(n)(\Psi-\Psi_\epsilon)(n_za_t)d\mu_N^{\normalfont\text{Leb}}(n)\leq \frac{\epsilon\cdot\mu_{N}^{\text{PS}}(f) }{m^{\bms}(\Gamma\bs G)}. $$

Define $\widetilde{\Psi}_\epsilon(g):=\min\{\Psi_\epsilon(g), \widetilde{\Psi}(g)\}$.  We have $\widetilde{\Psi}_\epsilon\in C_c(\Gamma\bs G)$, so that $(f,\widetilde{\Psi}_\epsilon)\in\mathcal W$, or \par 
\al{\label{0520}
 \lim_{t\rightarrow\infty}e^{(2-\delta)t}\int_{\Gamma\bs\Gamma N}f(n)\widetilde{\Psi}_\epsilon(n_za_t)d\mu_N^{\normalfont\text{Leb}}(n)=\frac{m^{\br}(\widetilde{\Psi}_\epsilon)\mu_{N}^{\text{PS}}(f) }{m^{\bms}(\Gamma\bs G)}. 
}

We also have 
\al{\label{0528}
\int_{\Gamma\bs G}(\widetilde{\Psi}(g)-\widetilde{\Psi}_\epsilon(g))dm^{\br}(g)\leq \int_{\Gamma\bs G}(\Psi(g)-\Psi_\epsilon(g))dm^{\br}(g)<\epsilon,
}
and

\al{\label{0521}\nonumber
&\hspace{5mm} \limsup_{t\rightarrow\infty}e^{(2-\delta)t}\int_{\Gamma\bs\Gamma N}f(n)(\widetilde{\Psi}-\widetilde{\Psi}_\epsilon)(n_za_t)d\mu_N^{\normalfont\text{Leb}}(n) \\
\nonumber&\leq \lim_{t\rightarrow\infty}e^{(2-\delta)t}\int_{\Gamma\bs\Gamma N}f(n)(\Psi-\Psi_\epsilon)(n_za_t)d\mu_N^{\normalfont\text{Leb}}(n) \\
&\leq \frac{ \epsilon\cdot\mu_{N}^{\text{PS}}(f) }{m^{\bms}(\Gamma\bs G)}.
}

Combining \eqref{0528}, \eqref{0521} and \eqref{0520}, and letting $\epsilon\rightarrow 0$, we obtain
 \al{\lim_{t\rightarrow\infty}e^{(2-\delta)t}\int_{\Gamma\bs\Gamma N}f(n)\widetilde{\Psi}(n_za_t)d\mu_N^{\normalfont\text{Leb}}(n)=\frac{m^{\br}(\widetilde{\Psi})\mu_{N}^{\text{PS}}(f) }{m^{\bms}(\Gamma\bs G)},}
 so that $(f,\widetilde{\Psi})\in\mathcal W$.

 \end{proof}

\subsection{Finishing the proof of Theorem \ref{07141}\label{1320}} We begin with an elementary geometric observation, which implies that any pair of points from $\mathcal{C}_t$ can not get too close. \par
\noindent $\bd{Observation}:$  For any two non-intersecting hemispheres based on $\CC$, the Euclidean distance of their apices $q_1,q_2$ satisfies 
\al{\label{2227}\vert\Re q_1-\Re q_2\vert \geq \Im q_1+\Im q_2.}
And from the hyperbolic distance formula,
\al{\label{2228}d(q_1,q_2)=\text{Arccosh}\left(1+\frac{|q_1-q_2|^2}{2\Im q_1\Im q_2}\right)\geq \text{Arccosh}\left(1+\frac{(\Im q_1+\Im q_2)^2}{2\Im q_1\Im q_2}\right)\geq \text{Arccosh} (3).}

From the observation \eqref{2227},  if $\bd{q}(g^{-1}\gamma_1),\bd{q}(g^{-1}\gamma_2)\in \Omega_i^*$ for $\gamma_1\neq\gamma_2\in \Gamma/\Gamma_H$, then $|\Re(q_1)-\Re(q_2)|\geq 2$.  For each $\gamma\in\Gamma/\Gamma_H$ with $\Im(\bd{q}(g^{-1}\gamma))>1$, place a circle of radius 1 centered at $\bd{q}_{\Re}(g^{-1}\gamma)$, then these circles are disjoint.  By an elementary packing argument, we have 
\al{\label{1440}
\#\bd{q}(g^{-1}\Gamma)\cap \Omega_i^* < \frac{\pi(\mathcal D(\Omega_i)+1)^2}{\pi}=(\mathcal D(\Omega_i)+1)^2.
}\\

The functions we are interested in are $f=\chi_{E}$ and $\Psi=F_{\bd{\Omega},\bd{r}}, F_{\bd{\Omega}}^{\boldsymbol{\beta}}$.

Suppose $r_j$ is a nonzero component of $\bd{r}$, then we have
\al{\label{1437}
F_{\bd{\Omega},\bd{r}}(g)=\prod_{1\leq i\leq k} \bd{1}\{\# \bd{q}(g^{-1}\Gamma/\Gamma_H)\cap\Omega_i^*=r_i\}\leq \#(\bd{q}(g^{-1}\Gamma/\Gamma_H)\cap \Omega_j^*),
}
and 
\al{\label{1441}
F_{\bd{\Omega}}^{\boldsymbol{\bd{\beta}}}(g)=\prod_{1\leq i\leq k} \#(\bd{q}(g^{-1}\Gamma/\Gamma_H)\cap \Omega_i^*)^{\beta_i}=\sum_{\bd{r}>0}\bd{r}^{\boldsymbol{\beta}}F_{\bd{\Omega},\bd{r}},
}
where $\bd{r}^{\boldsymbol{\beta}}(g)=\prod_{1\leq i\leq k}r_i^{\beta_i}$, and $\bd{r}>0$ means all components of $\bd{r}$ are nonnegative, and at least one component of $\bd{r}$ is positive. 

We notice that the right hand side of \eqref{1437} is of the form $\Psi_0$ (see \eqref{2242}), and the rightmost sum in \eqref{1441} is a finite sum because of \eqref{1440}. So both $F_{\boldsymbol{\Omega},\bd{r}}$ and $F_{\bd{\Omega}}^{\boldsymbol{\beta}}$ are dominated by (a finite linear combination of) $\Psi_0$.  Therefore, we can apply Proposition \ref{1426} to $f=\chi_E,\Psi=F_{\boldsymbol{\Omega},\bd{r}},F_{\bd{\Omega}}^{\boldsymbol{\beta}}$, once we have verified that $m^{\br}(\overline{\mathcal Disc (F_{\bd{\Omega},\bd{r}})})$, $m^{\br}(\overline{\mathcal Disc (F_{\bd{\Omega}}^{\boldsymbol{\beta}})})=0$. It is enough to show $m^{\br}(\overline{\mathcal Disc (F_{\bd{\Omega},\bd{r}})})=0$.   \\

Let $\mathcal M_{\Omega_i}:=\{g\in G: \bd{q}(g^{-1})\in\partial \Omega_i^*\}$. Using the $NAH$ decomposition, we can see that $\mathcal M_{\Omega_i}$ is a closed submanifold of $G$ of codimension 1, thus $m^{\br}(\mathcal M_{\Omega_i})=0$.  \par


Next, we show that
\begin{lem}\label{1537}The immersion $\mathcal M_{\Omega_i}\rightarrow \pi_1(\mathcal M_{\Omega_i})$ is proper: for each $g\in \mathcal{M}_{\Omega_i}$, there does not exist infinitely many $\gamma_j\in \Gamma_H\bs\Gamma$, $g_j\in \mathcal M$, $1<j<\infty$, such that $\lim_{j\rightarrow\infty}\gamma_jg_j=g$. 
\end{lem}
\begin{proof}
We argue by contradiction.  Suppose there exist infinitely many $\gamma_j\in \Gamma_H\bs\Gamma$, $g_j\in \mathcal M$, $1\leq j<\infty$, such that $\lim_{j\rightarrow\infty}\gamma_jg_j=g$.  Since $\bd{q}$ is continuous, we have $\lim_{j\rightarrow\infty}\bd{q}(g_j^{-1}\gamma_j^{-1})=\bd{q}(g^{-1})$.  We note that $\bd{q}(g_j^{-1}\gamma_j^{-1})$ are apices from disjoint hemispheres.  Let $L:=\left\{z+r\bd{j}\in\mathbb H^3:z\in C(\bd{q}_{\Re}(g^{-1}), 1), r\in\left(\frac{\Im(\bd{q}(g^{-1}))}{2},\infty\right)\right\}$.  Then $\bd{q}(g^{-1})\in L$.  But \eqref{2227} implies that there can be at most $\frac{(1+\mathcal \Im(\bd{q}(g^{-1})))^2}{ \Im(\bd{q}(g^{-1})^2}$ many points in $L$.  Thus we have a contradiction.
\end{proof}

Lemma \ref{1537} implies that $\overline{\pi_1(\mathcal M_{\Omega_i})}=\pi_1(\mathcal M_{\Omega_i})$, so that $m^{\br}(\overline{\pi_1(\mathcal M_{\Omega_i})} )=0$.  Let $\mathcal M_{\bd{\Omega}}:=\cup_{i=1}^k \mathcal M_{\Omega_i}$. As a finite union of $\mathcal M_{\Omega_i}$, $\mathcal M_{\bd{\Omega}}$ is closed in $G$ and the immersion $\mathcal M_{\bd{\Omega}}\rightarrow \Gamma\bs\Gamma \mathcal M_{\bd{\Omega}}$ is proper, so that $ \Gamma\bs\Gamma \mathcal M_{\bd{\Omega}}$ is closed in $\Gamma\bs G$ and $m^{\br}(\Gamma\bs \Gamma \mathcal M_{\bd{\Omega}})=0$.  \par

Our next lemma shows that $F_{\bd{\Omega},\boldsymbol{r}}$ is continuous outside $\Gamma\bs \Gamma \mathcal M_{\bd{\Omega}}$, and as a corollary, $m^{\br}(\overline{\mathcal Dist(F_{\bd{\Omega},\boldsymbol{r}})}), m^{\br}(\overline{\mathcal Dist(F_{\bd{\Omega}}^{\boldsymbol{\beta}})})\leq m^{\br}(\Gamma\bs \Gamma\mathcal M_{\bd{\Omega}}) =0$, whence we can obtain Theorem \ref{07141} by applying Proposition \ref{1426} with $f=\chi_E$, $\Psi=\Psi_0$, $\widetilde{\Psi}=F_{\bd{\Omega},\boldsymbol{r}}, F_{\bd{\Omega}}^{\boldsymbol{\beta}}$.

\begin{lem} Let $\mathcal M_{\bd{\Omega}}=\cup_{i=1}^k \mathcal M_{\Omega_i}$, then the function $F_{\bd{\Omega},\boldsymbol{r}}$ is continuous in $\Gamma\bs G-\Gamma\bs\Gamma \mathcal M_{\bd{\Omega}}$. 
\end{lem}
\begin{proof}
Since the immersion $\mathcal M_{\bd{\Omega}}\rightarrow \Gamma\bs\Gamma \mathcal M_{\bd{\Omega}}$ is proper,  for any $g\in G-\Gamma\mathcal M_{\bd{\Omega}}$, there exists a simply connected open neighborhood $O_g\subset G$ of $g$ such that $O_g\cap \Gamma \mathcal M_{\bd{\Omega}}=\emptyset$.  We claim that $F_{\bd{\Omega}}^{\boldsymbol{\beta}}$ is constant on $\Gamma\bs\Gamma O_g$, by showing that for each $1\leq i\leq k$ and each $\gamma\in \Gamma_H\bs\Gamma$, $\bd{1}\{ \bd{q}((\gamma g)^{-1})\in\Omega_i^* \}$ is constant in $O_g$.  We argue by contradiction.  Suppose $\bd{1}\{ \bd{q}((\gamma g)^{-1})\in\Omega_i^* \}$ is not constant in $O_g$, then there exists $g_1,g_2\in O_g$ such that $\bd{q}(( \gamma g_1)^{-1})\in\Omega_i^* $ and $ \bd{q}(( \gamma g_2)^{-1})\not\in\Omega_i^* $.  We observe that 
$$\overline{\Omega_{i}^* }\cap \overline{ G-\Omega_{i}^*}=\partial\Omega_i\cup\{\infty\}.$$
Let $p:[0,1]\rightarrow O_g$ be a path with $p(0)=g_1$ and $p(1)=g_2$.  Then for some $s\in(0,1]$, we have $\bd{q}((\gamma p(s))^{-1})\in \partial\Omega_i^*\cup\{\infty\}$.  If $\bd{q}((\gamma p(s))^{-1})\in \partial \Omega_i^*$, then $p(s)\in O_g\cap \Gamma\mathcal M_{{\Omega}_i}$, violating $\Omega_g\cap \Gamma\mathcal M_{\bd{\Omega}}=\emptyset$.  Thus $\bd{q}((\gamma p(s))^{-1})=\infty$ and we let $$s_0=\inf\{s\in(0,1]: \bd{q}((\gamma p(s))^{-1})=\infty\}.$$  By the continuity of $p$ and $\bd{q}$, we have $ \bd{q}((\gamma p(s_0))^{-1})=\infty$.  By the definition of $s_0$, for $s<s_0$, we have $\bd{q}((\gamma p(s))^{-1})\in \Omega_j^*$. Therefore, as $\Omega_j$ is bounded, $\bd{q}_{\Re}((\gamma p(s))^{-1})$ is bounded for $s\in(0,s_0)$.  But $\bd{q}_{\Re}((\gamma p(s_0))^{-1})=\infty$, and this is impossible as $\bd{q}_{\Re}$ is a continuous map.  Thus we arrived at a contradiction. \par
Therefore, for any $\gamma$, the function $\bd{1}\{ \bd{q}((\gamma g)^{-1}\in\Omega_i^* \}$ is constant in $O_g$ as desired.  This implies $F_{\bd{\Omega},\boldsymbol{r}}$ is constant in $\Gamma\bs\Gamma O_g$.  The lemma is thus proved.

\par

\end{proof}

\section{Pair correlation and nearest neighbor spacing}\label{pairnear}
In this section we deduce Theorem \ref{pairthm} (limiting pair correlation) and Theorem \ref{nearthm} (limiting nearest neighbor spacing) from Theorem \ref{07141}.  We give full detail for the limiting pair correlation; the proof for the limiting nearest neighbor spacing is similar and we give a sketch.  

\subsection{Pair correlation} The purpose of this section is to prove Theorem \ref{pairthm}.  Let $E\subset \CC$ be an open set with $E\cap \mathcal P\neq\emptyset$ and $\partial E$ empty or piecewise smooth.  The pair correlation function $P_{E,t}(\xi)$ on the set $\mathcal C_t$ is defined as
\al{
P_{E,t}(\xi)=\frac{1}{2\#\{\mathcal C_t\cap E\}}\sum_{\substack{p, q\in\mathcal C_t\cap E\\q\neq p}}\bd{1}\{|p-q|<e^{-t}\xi\}.
}

Let $B_r$ be the disk in $\CC$ centered at 0 with radius $r$.  We analyze the pair correlation function $P_{E,t}$ via the following mixed 1-moment function $P_{E,t,\epsilon}$:
\al{\label{0500}
P_{E,t,\epsilon}(\xi):=\frac{e^{2t}}{2\pi\epsilon^2\cdot\#\{\mathcal C_t\cap E\}}\int_{\CC}\chi_E(z)\mathcal{N}_t(B_{\epsilon},z)\mathcal N_t(B_\xi,z)dz-\frac{1}{2}.
}
Here $\epsilon$ is taken as a small enough positive number, say $\epsilon<\min\{\frac{1}{10},\frac{\xi}{10}\}$, and \eqref{2227} implies that $\mathcal N_t(B_\epsilon, z)\leq 1, \forall z\in \mathbb C$.\par
The function $P_{E,t,\epsilon}$ is an approximate to $P_{E,t}$.  Indeed, 

\al{\nonumber
&\int_{\CC}\chi_E(z)\mathcal{N}_t(B_{\epsilon},z)\mathcal N_t(B_\xi,z)dz\\
\nonumber=&\sum_{p\in\mathcal C_t}\int_{\CC}\bd{1}\{z\in e^{-t}B_{\epsilon}+p\}\mathcal N_t(B_\xi,z)\chi_E(z)dz\\
\nonumber\leq& \sum_{p\in\mathcal C_t}e^{-2t}\pi \epsilon^2 \mathcal N_t(B_{\xi+\epsilon},p)\chi_{E_{\epsilon^+}}(p)\\\nonumber
=&e^{-2t}\pi \epsilon^2  \sum_{p\in\mathcal C_t}\mathcal\chi_{E_{\epsilon^+}}(p)+e^{-2t}\pi \epsilon^2  \sum_{p\in\mathcal C_t}\mathcal\chi_{E_{\epsilon^+}}(p)\sum_{\substack{q\in\mathcal C_t\\q\neq p}}\bd{1}\{|q-p|<e^{-t}(\xi+\epsilon)\}\\
\leq& e^{-2t}\pi\epsilon^2\#(\mathcal C_t\cap E_{\epsilon^+})+e^{-2t}\pi\epsilon^2\sum_{p\in\mathcal C_t\cap E_{\epsilon^+}}\sum_{\substack{q\in\mathcal C_t\\q\neq p}}\bd{1}\{|q-p|<e^{-t}(\xi+\epsilon)\}.
\label{0459}
}

Putting \eqref{0459} back to \eqref{0500}, we have
\al{\label{pairineq1}
P_{E,t,\epsilon}(\xi)\leq \frac{\#(\mathcal C_t\cap E_{\epsilon^+})}{2\#(\mathcal C_t\cap E)}-\frac{1}{2}+ \frac{\#(\mathcal C_t\cap E_{\epsilon^+})}{\#(\mathcal C_t\cap E)}P_{E_{\epsilon^+},t}(\xi+\epsilon).
}
Similarly, we have
\al{\label{pairineq2}
P_{E,t,\epsilon}(\xi)\geq \frac{\#(\mathcal C_t\cap E_{\epsilon^-})}{2\#(\mathcal C_t\cap E)}-\frac{1}{2}+ \frac{\#(\mathcal C_t\cap E_{\epsilon^-})}{\#(\mathcal C_t\cap E)}P_{E_{\epsilon^-},t}(\xi-\epsilon).
}
We can work out from \eqref{pairineq1} and \eqref{pairineq2} that
\al{\label{0917}
P_{E,t}(\xi)\leq \frac{\#(\mathcal C_t\cap E_{\epsilon^+})}{\#(\mathcal C_t\cap E)}P_{E_{\epsilon^+},t,\epsilon}(\xi+\epsilon)+\frac{\#(\mathcal C_t\cap E_{\epsilon^+})}{2\#(\mathcal C_t\cap E)}-\frac{1}{2}
}
and
\al{\label{0918}
P_{E,t}(\xi)\geq \frac{\#(\mathcal C_t\cap E_{\epsilon^-})}{\#(\mathcal C_t\cap E)}P_{E_{\epsilon^-},t,\epsilon}(\xi-\epsilon)+\frac{\#(\mathcal C_t\cap E_{\epsilon^-})}{2\#(\mathcal C_t\cap E)}-\frac{1}{2}.
}

Letting $t\rightarrow\infty$ and then $\epsilon\rightarrow 0^+$ in \eqref{0917} and \eqref{0918},  Theorem \ref{pairthm} is proved once we have shown
\al{\label{1002}
\lim_{\epsilon\rightarrow 0^+}\lim_{t\rightarrow \infty}P_{E_{\epsilon^+},t,\epsilon}(\xi+\epsilon)=\lim_{\epsilon\rightarrow 0^+}\lim_{t\rightarrow \infty}P_{E_{\epsilon^-},t,\epsilon}(\xi-\epsilon)=P(\xi)
}
for some continuously differentiable function $P(\xi)$.\par

Now we analyze the limit of $P_{E,t,\epsilon}(\xi)$, as $t\rightarrow \infty$.  From Theorem \ref{circlecounting}, we have 
\al{\label{1454}
\lim_{t\rightarrow\infty}\frac{\# \mathcal C_t\cap E}{e^{\delta t}}=\frac{\mu_{H}^{\text{PS}}(\Gamma_H\bs H)w(E)}{\delta\cdot m^{\bms}(\Gamma\bs G)}.
}

From Theorem \ref{07141}, we have
\al{\label{0643}\nonumber
&\lim_{t\rightarrow\infty}e^{(2-\delta)t}\int_{\CC}\chi_E(z)\mathcal N_t(B_\epsilon,x)\mathcal N_t(B_\xi,x)dx\\\nonumber
=&\frac{w(E)}{m^{\bms}(\Gamma\bs G)}\cdot \int_{\Gamma\bs G}\left(\sum_{\gamma_1\in\Gamma_H\bs\Gamma}\bd{1}\{\bd{q}((\gamma g)^{-1})\in B_{\epsilon}\}\right)\cdot \#\{\gamma\in\Gamma_{H}\bs\Gamma: \bd{q}((\gamma g)^{-1})\in B_{\xi}^*\}dm^{\br}(g)\\
=&\frac{w(E)}{m^{\bms}(\Gamma\bs G)}\cdot \int_{\Gamma_H\bs G}\bd{1}\{\bd{q}(g^{-1})\in B_\epsilon^{\infty}\}\cdot \#\{\gamma\in\Gamma_{H}\bs\Gamma: \bd{q}((\gamma g)^{-1})\in B_{\xi}^*\}dm^{\br}(g).
}
Writing $g=ha_tn_z$ in the $HAN$ decomposition, from Proposition \ref{0130}, we have
\al{
\eqref{0643}=\frac{w(E)}{m^{\bms}(\Gamma\bs G)}\cdot\int_{\gamma_H\bs H}\int_{z\in B_\epsilon}\int_{0}^\infty e^{-\delta t}\#\{\gamma\in \Gamma_H\bs\Gamma:\bd{q}(n_{-z}a_{-t}h^{-1}\gamma^{-1})\in B_\xi^*\}dtdzd\mu_{H}^{\text{PS}}(h).
}

The conditions $z\in B_\epsilon$ and
$\bd{q}(n_{-z}a_{-t}h^{-1}\gamma^{-1})\in B_\xi^*$
imply that 
$\bd{q}(a_{-t}h^{-1}\gamma^{-1})\in B_{\xi+\epsilon}^*$.
Therefore, we have 
\al{\label{2155}
\eqref{0643}\leq \frac{\pi\epsilon^2w(E)}{m^{\bms}(\Gamma\bs G)}\cdot\int_{\Gamma_H\bs H}\int_{0}^\infty e^{-\delta t}\#\{\gamma\in \Gamma_H\bs\Gamma:\bd{q}(a_{-t}h^{-1}\gamma^{-1})\in B_{\xi+\epsilon}^*\}dt d\mu_{H}^{\text{PS}}(h)
}
and similarly,
\al{\label{2156}
\eqref{0643}\geq \frac{\pi\epsilon^2w(E)}{m^{\bms}(\Gamma\bs G)}\cdot\int_{\Gamma_H\bs H}\int_{0}^\infty e^{-\delta t}\#\{\gamma\in \Gamma_H\bs\Gamma:\bd{q}(a_{-t}h^{-1}\gamma^{-1})\in B_{\xi-\epsilon}^*\}dt d\mu_{H}^{\text{PS}}(h).
}

Define
\al{\nonumber
P(\xi):=&\frac{\delta}{2\mu_{H}^{\text{PS}}(\Gamma_H\bs H)}\int_{\Gamma_H\bs H}\int_{0}^\infty e^{-\delta t}\#\{\gamma\in \Gamma_H\bs\Gamma:\bd{q}(a_{-t}h^{-1}\gamma^{-1})\in B_{\xi}^*\}dt d\mu_{H}^{\text{PS}}(h)-\frac{1}{2}\\
=&\frac{\delta}{2\mu_{H}^{\text{PS}}(\Gamma_H\bs H)}\int_{\Gamma_H\bs H}\int_{0}^\infty e^{-\delta t}\#\{\gamma\in \Gamma_H\bs(\Gamma-\Gamma_H):\bd{q}(a_{-t}h^{-1}\gamma^{-1})\in B_{\xi}^*\}dt d\mu_{H}^{\text{PS}}(h).
\label{1641}
}

Combining \eqref{1454}, \eqref{0643}, \eqref{2155}, \eqref{2156}, we obtain
\al{\label{0958}
P(\xi-\epsilon)\leq \liminf_{t\rightarrow\infty}P_{E,t,\epsilon}(\xi)\leq \limsup_{t\rightarrow\infty}P_{E,t,\epsilon}(\xi)\leq P(\xi+\epsilon).
}
The definition of $P$ is independent of the set $E\subset \mathbb C$, so \eqref{0958} also holds with $E$ replaced by $E_{\epsilon^\pm}$.  Thus the relation \eqref{1002} is established once we have shown $P$ is continuously differentiable.

First, we observe that $P(\xi)$ is indeed finite, as $\mu_H^{\ps}(\Gamma_H\bs H)$ is finite and the integrand of \eqref{1641} is bounded: for each fixed $h$ and $t$, from \eqref{1440} we have 
$$\#\{\gamma\in\Gamma_H\bs\Gamma: \bd{q}(a_{-t}h^{-1}\gamma^{-1})\in B_{\xi}^{*}\}\leq (2\xi+1)^2.$$ 

Next, we show that the pair correlation function $P$ is continuously differentiable. \par

We observe that if there exists $t>0$ such that $\bd{q}(a_{-t}h^{-1}\gamma^{-1})\in B_{\xi}^*$, then $\bd{q}(h^{-1}\gamma^{-1})\in \mathfrak{C}_\xi \cup B_{\xi}^{*}$, where $\mathfrak C_\xi$ is the cone defined at \eqref{1010}.

We thus write $P(\xi)$ into two parts:
\al{\label{1120}\nonumber
P(\xi)=&\frac{\delta}{2\mu_{H}^{\text{PS}}(\Gamma_H\bs H)}\int_{\Gamma_H\bs H}\sum_{\substack{\gamma\in\Gamma_H\bs(\Gamma-\Gamma_H)\\\bd{q}(h^{-1}\gamma^{-1})\in B_{\xi}^* }}\int_{0}^\infty e^{-\delta t}\bd{1}\{\bd{q}(a_{-t}h^{-1}\gamma^{-1})\in B_{\xi}^*\}dt d\mu_{H}^{\text{PS}}(h)\\
 \nonumber=&\frac{1}{2\mu_{H}^{\text{PS}}(\Gamma_H\bs H)}\int_{\Gamma_H\bs H}\sum_{\gamma\in^\Gamma_H\bs (\Gamma-\Gamma_H)}  \bd{1}\{\bd{q}(h^{-1}\gamma^{-1})\in B_\xi^*\} \left(1-\left(\frac{\vert\bd{q}_{\Re}(h^{-1}\gamma^{-1})\vert}{\xi}\right)^\delta\right) \\
&+ \bd{1}\{\bd{q}(h^{-1}\gamma^{-1})\in \mathfrak C_\xi \} \left(\Im(\bd{q}(h^{-1}\gamma^{-1}))^\delta-\left(\frac{\vert\bd{q}_{\Re}(h^{-1}\gamma^{-1})\vert}{\xi}\right)^\delta\right) d\mu_H^{\ps}(h).
}
To proceed, we need the following lemma: 
\begin{lem}Define $$p(h,\xi)=\sum_{\gamma\in\Gamma_H\bs(\Gamma-\Gamma_H)}\bd{1}\{\bd{q}(h^{-1}\gamma^{-1})\in B_\xi^*\} +  \bd{1}\{\bd{q}(h^{-1}\gamma^{-1})\in \mathfrak C_\xi \} \Im(\bd{q}(h^{-1}\gamma^{-1}))^\delta.$$
Fixing $\xi$, then $p(h,\xi)$ is bounded for $h\in\Gamma_H\bs H$. \end{lem}

\begin{proof} First, from \eqref{1440}, we have
\al{
\sum_{\gamma_H\bs\Gamma-\Gamma_H}\bd{1}\{\bd{q}(h^{-1}\gamma^{-1})\in B_\xi^*\}<(2\xi+1)^2.
}
Next, let $\mathfrak C_\xi^{t_1,t_2}$ be the truncated cone
$$\mathfrak C_\xi^{t_1,t_2}:=\{z+r\bd{j}\in \HH^3: \frac{r}{|z|}>\frac{1}{\xi}, t_1<r\leq t_2 \}.$$
Recall the definition of $\mathfrak C_\xi$ at \eqref{1010}.  An elementary exercise in hyperbolic geometry shows that the 2-neighborhood of $\mathfrak C_\xi$ (the set of all points in $\HH^3$ having hyperbolic distance $<2$ to $\mathfrak C_\xi$) is contained in the cone
\al{\label{1009}\widetilde{\mathfrak C}_\xi:= \left\{z+r\bd{j}\in \HH^3: \frac{r}{|z|}>\frac{1}{e^2\xi}, 0<r\leq e^2 \right\},}
and the 2-neighborhood of $\mathfrak C_\xi^{t_1,t_2}$ is contained in the truncated cone 
$$\widetilde{\mathfrak C}_\xi^{t_1,t_2}:=\left\{z+r\bd{j}\in \HH^3: \frac{r}{|z|}>\frac{1}{e^2\xi}, \frac{t_1}{e^2}<r\leq t_2e^2 \right\}.$$
Therefore, for each $0<t<1$,
\al{\nonumber
&\sum_{\gamma\in \Gamma_H\bs (\Gamma-\Gamma_H)} \bd{1}\{\bd{q}(h^{-1}\gamma^{-1})\in \mathfrak C_\xi \} \Im(\bd{q}(h^{-1}\gamma^{-1}))^\delta\\\nonumber
 = &\sum_{n=0}^{\infty}   \sum_{\gamma\in \Gamma_H\bs (\Gamma-\Gamma_H)}\bd{1}\{\bd{q}(h^{-1}\gamma^{-1})\in \mathfrak C_\xi^{\frac{1}{2^{n+1}},\frac{1}{2^n}}\} \Im(q(h^{-1}\gamma^{-1}))^{\delta}  \\\label{12581}
\leq & \sum_{n=0}^{\infty}   \frac{\mathcal{V}ol(\widetilde{\mathfrak C}_\xi^{\frac{1}{2^{n+1}}, \frac{1}{2^{n}}})}{4\pi}\frac{1}{2^{n\delta}}\\
\label{12580}\ll&\sum_{n=0}^\infty \frac{1}{2^{n\delta}}<\infty,
}
where in \eqref{12581} we used a packing (by hyperbolic balls) argument combined with \eqref{2228}, and here $\mathcal{V}ol(\widetilde{\mathfrak C}_\xi^{\frac{1}{2^{n+1}},\frac{1}{2^{n}}})$ is the hyperbolic volume of $\widetilde{\mathfrak C}_\xi^{\frac{1}{2^{n+1}},\frac{1}{2^n}}$.  
\end{proof}

Now we show that $P(\xi)$ is differentiable.  For small $\epsilon>0$,
\al{\nonumber
&\frac{P(\xi+\epsilon)-P(\xi)}{\epsilon}=\\
\nonumber
&\frac{1}{2\mu_H^{\text{PS}}(\Gamma_H\bs H)}\int_{h\in\Gamma_H\bs H}\sum_{\gamma\in\gamma_H\bs(\Gamma-\Gamma_H)}\bd{1}\{\bd{q}(h^{-1}\gamma^{-1})\in B_{\xi}^*\cup \mathfrak C_\xi \}\cdot\frac{\vert\bd{q}_{\Re}(h^{-1}\gamma^{-1})\vert^{\delta}\left(\frac{1}{\xi^{\delta}}-\frac{1}{(\xi+\epsilon)^{\delta}}  \right)}{\epsilon}d\mu_{H}^{\ps}(h)\\\nonumber
+& \frac{1}{2\mu_H^{\text{PS}}(\Gamma_H\bs H)}\int_{h\in\Gamma_H\bs H}\sum_{\gamma\in\gamma_H\bs(\Gamma-\Gamma_H)}\bd{1}\{\bd{q}(h^{-1}\gamma^{-1})\in (B_{\xi+\epsilon}^*-B_\xi^*)\}\cdot \frac{1-\left(\frac{\vert\bd{q}_{\Re}(h^{-1}\gamma^{-1})\vert}{\xi+\epsilon}\right)^\delta}{\epsilon} \\\nonumber
&\hspace{0.7in}+ \bd{1}\{\bd{q}(h^{-1}\gamma^{-1})\in (\mathfrak C_{\xi+\epsilon}-\mathfrak C_\xi)\} \cdot\frac{\Im(\bd{q}(h^{-1}\gamma^{-1}))^\delta-\left(\frac{\vert\bd{q}_{\Re}(h^{-1}\gamma^{-1})\vert}{\xi+\epsilon}\right)^\delta}{\epsilon} d\mu_{H}^{\ps}(h)\\\nonumber
=&\frac{\delta}{2\mu_H^{\text{PS}}(\Gamma_H\bs H)}\int_{h\in\Gamma_H\bs H}\sum_{\gamma\in\gamma_H\bs(\Gamma-\Gamma_H)}\bd{1}\{\bd{q}(h^{-1}\gamma^{-1})\in B_{\xi}^*\cup \mathfrak C_\xi \}\cdot\frac{\vert\bd{q}_{\Re}(h^{-1}\gamma^{-1})\vert^{\delta}}{\xi^{\delta+1}}(1+O_\xi(\epsilon))d\mu_{H}^{\ps}(h)\\
&+O_\xi\left(\int_{h\in\Gamma_H\bs H} p(h,\xi+\epsilon)-p(h,\xi) d\mu_H^{\ps}(h)\right).\label{0229}
}

Noting that
$$\sum_{\gamma\in\gamma_H\bs(\Gamma-\Gamma_H)}\bd{1}\{\bd{q}(h^{-1}\gamma^{-1})\in B_{\xi}^*\cup \mathfrak C_\xi \}\cdot\frac{\vert\bd{q}_{\Re}(h^{-1}\gamma^{-1})\vert^{\delta}}{\xi^{\delta+1}}\ll_\xi p(h,\xi),$$
and letting $\epsilon\rightarrow 0^{+}$, we have 
\al{\label{2230}\lim_{\epsilon\rightarrow 0^+}\frac{P(\xi+\epsilon)-P(\xi)}{\epsilon}=  \frac{\delta}{2\mu_H^{\text{PS}}(\Gamma_H\bs H)}\int_{h\in\Gamma_H\bs H}\sum_{\substack{\gamma\in\gamma_H\bs\Gamma-\Gamma_H\\\bd{q}(h^{-1}\gamma^{-1})\in B_{\xi}^*\cup \mathfrak C_\xi}}\frac{\vert\bd{q}_{\Re}(h^{-1}\gamma^{-1})\vert^{\delta}}{\xi^{\delta+1}}d\mu_{H}^{\ps}(h),}
once we have shown that the term $O(\cdot)$ from \eqref{0229} goes to 0 as $\epsilon\rightarrow 0^{+}$.   
Indeed, since $p(h,\xi)$ is bounded with respect to $h$ and monotone with respect to $\xi$, by Lebesgue's Dominated Convergence Theorem,
\al{\label{0253}&\nonumber\lim_{\epsilon\rightarrow 0^+}\int_{h\in\Gamma_H\bs H}p(h,\xi+\epsilon)-p(h,\xi) d\mu_{H}^{\ps}(h) \\\nonumber
=&\sum_{\gamma\in\gamma_H\bs\Gamma-\Gamma_H}\int_{h\in\Gamma_H\bs H}\bd{1}\{\bd{q}(h^{-1}\gamma^{-1})\in \partial ( B_\xi^*\cup \mathfrak C_\xi ) \}\cdot\max\{\Im (\bd{q}(h^{-1}\gamma^{-1}))^{\delta},1\}   d\mu_{H}^{\ps}(h).
}
We can check that for each $\gamma\in \Gamma_H\bs (\Gamma-\Gamma_H)$, the set $H_{\gamma,\xi}:=\{h\in H: \bd{q}(h^{-1}\gamma^{-1})\in\partial(B_\xi^{*}\cup \mathfrak C_\xi ) \}$ is contained in an algebraic subvariety of $H$ of codimension 1.  Therefore, $\mu_{H}^{\ps}(H_{\gamma,\xi})=0$, so that \eqref{0253}=0, and \eqref{2230} is established.  \\

By a similar consideration, we can also show 
\aln{
\lim_{\epsilon\rightarrow 0^+}\frac{P(\xi)-P(\xi-\epsilon)}{\epsilon}=  \frac{\delta}{2\mu_H^{\text{PS}}(\Gamma_H\bs H)}\int_{h\in\Gamma_H\bs H}\sum_{\substack{\gamma\in\gamma_H\bs\Gamma-\Gamma_H\\\bd{q}(h^{-1}\gamma^{-1})\in B_{\xi}^*\cup \mathfrak C_\xi}}\frac{\vert\bd{q}_{\Re}(h^{-1}\gamma^{-1})\vert^{\delta}}{\xi^{\delta+1}}d\mu_{H}^{\ps}(h).
}

Therefore, $P$ is differentiable. The continuity of $P'$ follows from that, by the Dominated convergence theorem, 
\al{\nonumber\limsup_{\epsilon\rightarrow 0^\pm}|P'(\xi+\epsilon)-P'(\xi)|\ll_\xi &\sum_{\gamma\in \Gamma_H\bs(\Gamma-\Gamma_H)}\int_{\Gamma_H\bs H}\bd{1}\{\bd{q}(h^{-1}\gamma^{-1})\in\partial (B_\xi^*\cup \mathfrak C_\xi)\}d\mu_{H}^{\ps}(h) \\
\nonumber
=&0.
}

Finally, the reason that $P$ is supported away from 0 is due to the elementary observation \eqref{2227}. Theorem \ref{pairthm} is thus completely proved.

\subsection{Nearest Neighbor Spacing}
As usual we let $E\subset \CC$ be an open set with no boundary or piecewise smooth boundary, and with $E\cap {P}\neq \emptyset$.  
For any $p\in\mathcal C_t$, let $d_t(p)=\min\{|p-q|: q\in\mathcal C_t, q\neq p\}$.
The nearest neighbor spacing function $Q_{E,t}$ is defined by
\al{
Q_{E,t}(\xi)=\frac{1}{\#\{\mathcal C_t\cap E\}}\sum_{\substack{p\in\mathcal C_t\cap E}}\bd{1}\{d_t(p)<e^{-t}\xi\}.
}
We sketch our analysis for $Q_{E,t}$, which is in a very similar fashion as we did for the pair correlation function.  The function $Q_{E,t}(\xi)$ can be approximated by the following function
\al{\nonumber
&Q_{E,t,\epsilon}(\xi):=\\
&1-\frac{e^{2t}}{\pi\epsilon^2\cdot\#\{\mathcal C_t\cap E\}}\int_\CC \chi_E(z)\bd{1}\left\{\#( a_{-t}n_z \bd{q}(\Gamma)\cap B_{\epsilon}^*)=1 \right\}\bd{1}\left\{ \#(a_{-t}n_z \bd{q}(\Gamma)\cap B_{\xi}^*)=1 \right\}dz.
}
Indeed, one can check that
\al{
Q_{E,t}(\xi-\epsilon)\leq Q_{E,t,\epsilon}(\xi)\leq Q_{E,t}(\xi+\epsilon).
}

Applying Theorem \ref{07141} to $Q_{E,t,\epsilon}$ and letting $\epsilon\rightarrow 0^+$,  we obtain Theorem \ref{nearthm}.  The continuity of $Q$ follows from that, by the Dominated Convergence Theorem, 
\al{\nonumber\limsup_{\epsilon\rightarrow 0^\pm}|Q(\xi+\epsilon)-Q(\xi)|\ll_\xi &\sum_{\gamma\in \Gamma_H\bs(\Gamma-\Gamma_H)}\int_{\Gamma_H\bs H}\int_0^\infty e^{-\delta t}\bd{1}\{\bd{q}(a_{-t}h^{-1}\gamma^{-1})\in\partial B_\xi^*\}dtd\mu_{H}^{\ps}(h) \\
\nonumber
=&0.
}

\bibliographystyle{plain}
\bibliography{fractals}

\end{document}